\documentclass[11pt]{article}
\usepackage{graphicx}
\usepackage{graphics}
\usepackage{amsfonts}
\usepackage{amscd}
\usepackage{amsthm}
\usepackage{indentfirst}
\usepackage[all]{xy}
\usepackage{amssymb, amsmath, amsthm, amsgen, amstext, amsbsy, amsopn}
\usepackage{epic,eepic}
\usepackage{hyperref}
\usepackage{enumitem}
\usepackage{color}
\usepackage{dsfont}
\usepackage[margin=1.3in]{geometry}
\usepackage{comment}

\usepackage{xcolor}

\usepackage{appendix}

\theoremstyle{plain}
\newtheorem*{thm*}{Theorem}
\newtheorem{thm}{Theorem}
\newtheorem{lemma}{Lemma}[section]
\newtheorem{prop}[lemma]{Proposition}
\newtheorem{claim}[lemma]{Claim}
\newtheorem{cor}[lemma]{Corollary}

\theoremstyle{definition}
\newtheorem{defn}[lemma]{Definition}
\newtheorem{rem}[lemma]{Remark}

\newcommand{\Qed}{\ \hfill \qedsymbol \bigskip}

\title{\textbf{New steps in $C^0$ symplectic and contact geometry of smooth submanifolds}}
\author{\textbf{Maksim Stoki\'c}}

\begin{document}

\maketitle

\begin{abstract}
    We provide a $C^0$ counterexample to the Lagrangian Arnold conjecture in the cotangent bundle of a closed manifold. Additionally, we prove a quantitative $h$-principle for subcritical isotropic embeddings in contact manifolds, and provide an explicit construction of a contact homeomorphism which takes a subcritical isotropic curve to a transverse one. On the rigid side, we give another proof of the Dimitroglou Rizell and Sullivan theorem \cite{RS22} which states that Legendrian knots are preserved by contact homeomorphisms, provided their image is smooth. Moreover, our method gives related examples of rigidity in higher dimensions as well.
\end{abstract}

\tableofcontents

\section{Introduction}

The celebrated Eliashberg-Gromov rigidity theorem states that a diffeomorphism which is a $C^0$-limit of symplectomorphisms is itself symplectic. Motivated by this, symplectic homeomorphisms are defined as $C^0$-limits of symplectomorphisms. The contact version of Eliashberg-Gromov theorem holds as well (see \cite{MuSp14},\cite{Mu19}), thus we can analogously define contact homeomorphisms as $C^0$-limits of contactomorphisms.\\

A powerful tool for proving flexibility statements in $C^0$ symplectic geometry is the \textit{quantitative $h$-principle} introduced first for symplectic 2-discs in \cite{BO16}, and more recently for subcritical isotropic discs in \cite{BO21}. However, contact quantitative $h$-principle has not been established previously. In this paper we show that a contact version of a quantitative $h$-principle for subcritical isotropic embeddings holds as well. A submanifold $\Lambda\subset Y$ of a contact manifold $(Y,\xi)$ is called \textit{isotropic} if $T\Lambda\subset\xi|_{\Lambda}$. Isotropic submanifolds are important objects in contact geometry. The dimension of an isotropic submanifold of an $(2n+1)$-dimensional contact manifold is always not greater than $n$. If the dimension is strictly less than $n$, the isotropic submanifold is called \textit{subctitical}, and otherwise it is called \textit{Legendrian}.

\begin{thm}[Quantitative $h$-principle for subcritical isotropic embeddings]\label{quant-h-princ-discs}
Let $V\subset(\mathbb{R}^{2n+1},\mathrm{ker}\,dz-\sum_{i=1}^ny_idx_i)$ be an open contact submanifold, and let $k<n$.
\begin{enumerate}[label=(\alph*)]\label{Theorem1}
    \item Let $u_0,u_1:D^k\rightarrow V$ be isotropic embeddings of the closed disc. Assume there exists a homotopy $F:D^k\times [0,1]\rightarrow V$ between $u_0$ and $u_1$ of size less than $\varepsilon$ ($\mathrm{diam}\,F(\{z\}\times[0,1])<\varepsilon$ for all $z\in D^k$). Then there exists a contact isotopy $(\Psi^t)_{t\in[0,1]}$ such that $\max_{t\in[0,1]}d_{C^0}(\Psi^t,\mathrm{Id})<\varepsilon$ and $\Psi^1\circ u_0=u_1$.
    \item Let $n\geq 2$, and let $\gamma_0,\gamma_1:S^1\rightarrow V$ be isotropic embeddings. We assume there exists a homotopy $\gamma_t:S^1\rightarrow V$ between $\gamma_0$ and $\gamma_1$ of size less than $\varepsilon$. Then there exists a contact isotopy $(\Psi^t)_{t\in[0,1]}$ such that $\max_{t\in[0,1]}d_{C^0}(\Psi^t,\mathrm{Id})<\varepsilon$ and $\Psi^1\circ \gamma_0=\gamma_1$.
\end{enumerate}
\end{thm}

\noindent Our arguments rely on the classical $h$-principle for subcritical isotropic embeddings in contact manifolds. In appendix we provide a proof of it based on the book \cite{EM02}. As an application of this version of quantitative $h$-principle, we show that isotropic curves are flexible. More precisely, we prove that every contact manifold of dimension at least 5 admits a contact homeomorphism which maps a closed isotropic curve to a transverse one.

\begin{thm}\label{Theorem2}
    Let $(V,\xi)$ be a contact manifold of dimension at least $5$. There exists an isotropic embedding $\gamma:S^1\rightarrow V$ and a contact homeomorphism $h:V\rightarrow V$ which takes $\gamma$ to the smooth transverse knot $h\circ\gamma$.
\end{thm}
\noindent Our proof relies on the quantitative $h$-principle. On the other hand, for non-closed embedded curves we provide an explicit construction in Section \ref{sect4.2} where the construction does not rely on the quantitative $h$-principle.\\

For the next result, we need to define Hamiltonian homeomorphisms of symplectic manifolds. We use the definition given by Muller and Oh in \cite{OM}.
\begin{defn}[Hamiltonian homeomorphisms]
    Let $(\phi^t)_{t\in[0,1]}$ be compactly supported isotopy of a symplectic manifold $(M,\omega)$. We say that $\phi^t$ is a \textbf{hameotopy}, or a continuous Hamiltonian flow, if there exists a compact set $K\subset M$ and a sequence of smooth Hamiltonians $H_i:[0,1]\times M\rightarrow\mathbb{R}$ supported in $K$ such that:
    \begin{enumerate}
        \item The sequence $\phi^t_{H_i}$ $C^0$-converges to $\phi^t$, uniformly in $t$, i.e. $\max_{t\in[0,1]}d_{C^0}(\phi^t_{H_i},\phi^t)\rightarrow 0$ as $i\rightarrow\infty$.
        \item The sequence of Hamiltonians $H_i$, converges uniformly to a continuous function $H:[0,1]\times M\rightarrow\mathbb{R}$, i.e. $||H_i-H||_{\infty}\rightarrow 0$ as $i\rightarrow\infty$.
    \end{enumerate}
    We say that $H$ generates $\phi^t$, denote $\phi^t=\phi^t_H$, and call $H$ a continuous Hamiltonian. A homeomorphism is called a \textbf{Hamiltonian homeomorphism} if it is the time-1 map of a continuous Hamiltonian flow. We will denote the set of all Hamiltonian homeomorphisms by $\mathrm{Hameo}(M,\omega)$.
\end{defn}

One of applications of the symplectic quantitative h-principle for curves is a $C^0$ counterexample to the Arnold conjecture for symplectic manifolds of dimension at least 4, namely, an existence of $\phi \in\mathrm{Hameo}(M,\omega)$ having only one fixed point, on every closed symplectic manifold $(M,\omega)$ of dimension at least 4 (see \cite{BHS}). The Arnold conjecture admits a version for Lagrangian intersections, and it is natural ask if the Lagrangian Arnold conjecture survives when we pass to a Hamiltonian homeomorphisms. We construct a counterexample using quantitative $h$-principle for curves in the cotangent bundle.

\begin{thm}
    Let $L$ be a closed manifold of dimension at least $2$, and $(T^*L,\omega_{\mathrm{std}})$ cotangent bundle with its canonical symplectic form. Then there exists a Hamiltonian homeomorphism $f\in\mathrm{Hameo}(T^*L)$ such that the image of the zero-section $f(L_0)$ intersects the zero-section $L_0$ at a single point.
\end{thm}

\begin{rem}
    The theorem has been proven in \cite{BHS2} when $\mathrm{dim}\, L=2$, but the proof does not generalize to higher dimensions since it uses the fact that every arc $\gamma\subset L$ which is homeomorphic to the interval $[0,1]$ admits a basis of neighbourhoods homeomorphic to open discs, which is only true in dimension $2$. When $\mathrm{dim}\, L=1$, we have $L=S^1$ and $T^*L\cong S^1\times\mathbb{R}$ is a cylinder, however Hamiltonian homeomorphisms preserve area so we must have at least $2$ intersections between the zero-section and its image.
\end{rem}

Finally, we give a new proof of $C^0$-rigidity of Legendrian knots, and explain a possible direction towards proving $C^0$-rigidity of Legendrian submanifolds in higher dimensions. More precisely, we try to answer the following question: \textit{Let $\psi:(Y,\xi)\rightarrow (Y,\xi)$ be a contact homeomorphism. Let $\Lambda\subset Y$ be Legendrian submanifold, such that $\psi(\Lambda)$ is smooth submanifold. Must $\psi(\Lambda)$ be Legendrian?}\\

Dimitroglou Rizell and Sullivan recently proved \cite{RS22} that Legendrian knot cannot be mapped to smooth a non-Legendrian knot via contact homeomorphism. Their proof relies on methods available only in the the 3-dimensional contact manifolds. On the other hand, the main tool for our proof is recent result of Entov and Polterovich \cite{EP21} which includes the concept of contact interlinking of Legendrians, and it is available in higher dimensions as well. This is the first higher dimensional instance of the above question as stated.

\begin{defn}
    Let $(Y^{2n+1},\xi)$ be a contact manifold, and $N\subset Y$ compact subset. We call $N$ \textit{nearly Reeb invariant} if every open neighbourhood $U\supset N$ contains an open subset $V\subset U$ (where $N\subset V$) such that $V$ is invariant under the Reeb flow of some contact form on $U$ associated with $\xi$.
\end{defn}

\begin{thm}\label{Theorem4}
    Image of a closed Legendrian submanifold $\Lambda\subset (Y,\xi)$ via contact homeomorphism cannot be nearly Reeb invariant.
\end{thm}

\begin{rem}
    The theorem remains true, with the same proof, for a larger class of maps than contact homeomorphisms. Namely, any continuous map $\varphi:Y\rightarrow Y$ which can be written as a $C^0$-limit of a sequence of contactomorphisms and satiesfies $\varphi(\Lambda)\cap\varphi(Y\setminus\Lambda)=\emptyset$ cannot map $\Lambda$ to a nearly Reeb invariant set.
\end{rem}

In contact $3$ manifolds, a neighbourhood of a transverse knot can be contactly embedded into the neighbourhood of the zero-section in $(S^1\times\mathbb{R}^2,d\theta-xdy)$. For any $\varepsilon>0$ the set $S^1\times D(\varepsilon)\subset S^1\times\mathbb{R}^2$ is invariant under the Reeb flow of $d\theta-xdy$, therefore transverse knots are nearly Reeb invariant. But we have even more:

\begin{prop}\label{prop1.3}
    Let $(Y,\xi)$ be a contact 3 manifold and $K\subset Y$ smooth non-Legendrian knot (there exists a point $p\in K$ such that $T_pK\pitchfork\xi_p$). Then $K$ is nearly Reeb invariant.
\end{prop}

\begin{cor}[\cite{RS22}]
    Let $\Lambda\subset(Y,\xi)$ be a Legendrian knot and $\varphi:Y\rightarrow Y$ a contact homeomorphism. If the image $\varphi(\Lambda)$ is smooth, then it must be Legendrian.
\end{cor}

We now give an example of nearly Reeb invariant submanifold in dimension $2n+1>3$. Let $T\subset (M,\mathrm{ker}\,\lambda_M)$ be a transverse knot such that the Reeb vector field $R_{\lambda_M}$ is tangent to $T$. Let $(N,d\lambda_N)$ be an exact $(2n-2)$-dimensional symplectic manifold and $L\subset N$ submanifold of dimension $n-1$. Then $T\times L\subset (M\times N,\pi^*_M\lambda_M+\pi^*_N\lambda_N)$ is nearly Reeb invariant submanifold of dimension $n$.\\

\noindent\textbf{Acknowledgements.} I am very grateful to L. Buhovsky for his guidence, support and helpful discussions. I also thank G. Dimitroglou Rizell, L. Polterovich, S. Seyfaddini and M. Sullivan for the valuable feedback. This work was partially supported by ERC Starting Grant 757585 and ISF Grant 2026/17.

\section{Preliminaries}

\subsection{Symplectic and contact homeomorphisms}

Let $M$ be either a symplectic manifold with symplectic form $\omega$, or a contact manifold with the contact structure $\xi$. We equip $M$ with a Riemannian distance $d$. Given two compactly supported maps $\phi,\psi$ on $M$, we denote
$d_{C^0}(\phi,\psi)=\max_{x\in M}d(\phi(x),\psi(x))$. We say that a sequence of maps $\phi_i:M\rightarrow M$ converges uniformly, or $C^0$-converges to $\phi$, if there exists a compact subset of $M$ which contains the supports of all $\phi_i$'s, and moreover $d_{C^0}(\phi_i,\phi)\rightarrow 0$ as $i\rightarrow\infty$. The notion of $C^0$-convergence does not depend on the choice of a Riemannian metric.

\begin{defn}
    A homeomorphism $\theta:M\rightarrow M$ is said to be \textbf{symplectic/contact} if it is a $C^0$-limit of a sequence of symplectic/contact diffeomorphisms.
\end{defn}

The Eliashberg-Gromov rigidity theorem (which holds in the contact case as well) implies that a smooth symplectic/contact homeomorphism is itself symplectic/contact, i.e. preserves symplectic/contact structure. The next lemma will be useful for constructing homeomorphisms as a $C^0$-limits of diffeomorphisms.

\begin{lemma}\label{homeomorphismLemma}
    Let $f_0,f_1:N\rightarrow M$ be topological embeddings of a compact manifold $N$ (possibly with boundary). Let $\{U_i\}_{i\geq 1}$ be decreasing sequence of open sets such that $\bigcap_{i\geq 1}U_i=f_0(N)$. Let $\{\psi_i\}_{i\geq 1}$ be a sequence of compactly supported diffeomorphisms of $M$ which satisfies
    \begin{enumerate}
        \item $\psi_i$ is supported inside $\varphi_{i-1}(U_i)\supset f_1(N)$, where $\varphi_i=\psi_{i}\circ\psi_{i-1}\circ\cdots\circ\psi_1$,
        \item $\sum_{i=1}^{\infty}d_{C^0}(\psi_i,\mathrm{Id})<\infty$,
        \item The sequence $\varphi_i\circ f_0$ $C^0$-converges to $f_1$.
    \end{enumerate}
    Then the sequence $\varphi_i$ $C^0$-converges to a homeomorphism $\varphi:M\rightarrow M$, so that $\varphi\circ f_0=f_1$.
\end{lemma}

\begin{proof}
    From 2. we get that $\{\varphi_i\}$ is a Cauchy sequence, so it converges uniformly to a continuous map $\varphi$. Property 3. implies $\varphi\circ f_0=f_1$. Let us prove that $\varphi$ is an injective map, and hence a homeomorphism. Let $x\neq y$ be different points on $M$.
    \begin{itemize}
        \item If $x,y\in f_0(N)$, then 3. implies that $\varphi(x)=f_1\circ f_0^{-1}(x)\neq f_1\circ f_0^{-1}(y)=\varphi(y)$.
        \item If $x,y\notin f_0(N)$, then $x,y\in U^c_i$ for $i\geq i_0$, and hence 1. implies that $\varphi_{i_0}(x)=\varphi_{i_0+1}(x)=\ldots=\varphi(x)$ and $\varphi_{i_0}(y)=\varphi(y)$ and so $\varphi(x)\neq\varphi(y)$.
        \item If $x\in f_0(N)$ and $y\notin f_0(N)$, then as before we have $\varphi_{i}(y)=\varphi(y)$ for $i\geq i_0$. However, $\varphi(y)=\varphi_{i_0}(y)\notin\varphi_{i_0}(U_{i_0+1})\supset f_1(N)$ hence we have $\varphi(y)\notin f_1(N)$. On the other hand $\varphi(x)\in f_1(N)$ and therefore $\varphi(x)\neq\varphi(y)$.
    \end{itemize}
\end{proof}

\subsection{Quantitative $h$-principle in symplectic geometry}

Useful tool in proving flexibility statements in $C^0$-symplectic topology is quantitative $h$-principle introduced first in \cite{BO16}. It is an extension of various types of Gromov's $h$-principle, and so far it has been proven for symplectic 2-discs in \cite{BO16}, curves \cite{BHS} and most recently for subcritical isotropic discs \cite{BO21}. We will need the following quantitative $h$-principle for curves in the construction of counterexample to the Lagrangian Arnold conjecture.

\begin{thm}[Quantitative $h$-principle for curves \cite{BHS}]\label{h-principleCurves}
    Let $(M,\omega)$ be a symplectic manifold of dimension at least $4$, and let $\varepsilon>0$. Suppose that $\gamma_0,\gamma_1:[0,1]\rightarrow M$ are two smoothly embedded curves such that
    \begin{itemize}
        \item $\gamma_0$ and $\gamma_1$ coincide near $t=0$ and $t=1$,
        \item there exists a homotopy relative to the end points, between $\gamma_0$ and $\gamma_1$ under which the trajectory of any point of $\gamma_0$ has diameter less than $\varepsilon$, and the symplectic area of the element of $\pi_2(M,\gamma_1\#\overline{\gamma_0})$ defined by this homotopy has area $0$.
    \end{itemize}
    Then, for any $\rho>0$, there exists a compactly supported Hamiltonian $F$, generating a Hamiltonian isotopy $\varphi^s:M\rightarrow M,\,s\in[0,1]$ such that
    \begin{enumerate}
        \item $F$ vanishes near $\gamma_0(0)$ and $\gamma_0(1)$,
        \item $\varphi^1\circ\gamma_0=\gamma_1$,
        \item $d_{C^0}(\varphi^s,\mathrm{Id})<2\varepsilon$ for each $s\in[0,1]$, and $||F||_{\infty}\leq\rho$,
        \item $F$ is supported in the $2\varepsilon$-neighbourhood of the image of $\gamma_0$.
    \end{enumerate}
\end{thm}

\begin{rem}
    If $(M,\omega=d\lambda)$ is an exact symplectic manifold, the condition on the symplectic area of the $\pi_2(M,\gamma_1\#\overline{\gamma_0})$ is equivalent to $\int_0^1\gamma_0^*\lambda=\int_0^1\gamma_1^*\lambda$.
\end{rem}

\section{Contact quantitative $h$-principle for isotropic embeddings}

Let us introduce some notation and terminology. Let $A\subset M$ be a subset of a manifold $M$, then $Op(A)\subset M$ denotes sufficiently small open neighbourhood of $A$. Let $V\subset\mathbb{R}^{2n+1}$ be an open subset, $M$ a manifold of dimension $k<n$, and assume that the tangent bundle $TM$ is trivial. Fix a basis $(X_1,\ldots,X_k)$ of $TM$, and let $(Y_1,\ldots,Y_{2n+1})$ be a basis of $TV$, such that $\xi_V=\mathrm{span}(Y_1,\ldots,Y_{2n})$. Let $\mathcal{M}(k,2n+1)$ be a space of $(k,2n+1)$ matrices of rank $k$, and denote $G^{\mathrm{iso}}(k,2n+1)$ its subspace consisting of matrices which have all zeros in the $(2n+1)$-th row. A monomorphism $F:TM\rightarrow TV$ can be seen as the map $F :M\rightarrow\mathcal{M}(k,2n+1)$ by writing in $i$-th column the coordinates of $F(X_i)$ in the basis $(Y_1,\ldots,Y_{2n+1})$. Then isotropic condition can be written as $F(M)\subset G^{\mathrm{iso}}(k,2n+1)$.\\

If $A\subset M$, a homotopy $f:M\rightarrow G^{\mathrm{iso}}(k,2n+1)$ rel $Op(A)$ is a continuous map $F:[0,1]\times M\rightarrow G^{\mathrm{iso}}(k,2n+1)$ such that $F(t,x)=f(x)$ for $x\in Op(A)$. We say that a homotopy $G:[0,1]^2\times M\rightarrow G^{\mathrm{iso}}(k,2n+1)$ between $F_0,F_1:[0,1]\times M\rightarrow G^{\mathrm{iso}}(k,2n+1)$ is relative to $Op(A)$ and $\{0,1\}$ if $G(s,t,x)=F_0(t,x)=F_1(t,x)$ for all $(s,t,x)\in [0,1]^2\times Op(A)$, and $G(s,i,\cdot)=F_0(i,\cdot)$ for all $(s,i)\in[0,1]\times\{0,1\}$.\\

By abuse of notation we set $D^k=[-1,1]^k$ and consider the case $M=D^k$. Using relative $C^0$-dense version of the Theorem \ref{thm12.4.1}, we now formulate the $h$-principle for subcritical isotropic embeddings of discs.

\begin{thm}[Parametric $C^0$-dense relative $h$-principle for isotropic discs]\label{h-principleDiscs} Let $k<n$:
\begin{enumerate}[label=(\alph*)]
    \item Let $\rho:D^k\rightarrow\mathbb{R}^{2n+1}$ be an embedding whose restriction to a neighbourhood of a closed subset $A\subset D^k$ is isotropic. Assume that $d\rho$ is homotopic to a map $G:D^k\rightarrow G^{\mathrm{iso}}(k,2n+1)$ rel $Op(A)$. Then, for any $\varepsilon>0$, there exists an isotropic embedidng $u:D^k\rightarrow\mathbb{R}^{2n+1}$ such that $u|_{Op(A)}=\rho|_{Op(A)}$, $d_{C^0}(\rho,u)<\varepsilon$ and $du:D^k\rightarrow G^{\mathrm{iso}}(k,2n+1)$ is homotopic to $G$ rel $Op(A)$.
    \item Let $u_0,u_1:D^k\rightarrow(\mathbb{R}^{2n+1},\xi_{\mathrm{std}})$ be isotropic embeddings which coincide on a neighbourhood of a closed subset $A\subset D^k$. Let $G:[0,1]\times D^k\rightarrow G^{\mathrm{iso}}(k,2n+1)$ be a homotopy between $du_0,du_1$ rel $Op(A)$ and $\rho_t:D^k\rightarrow\mathbb{R}^{2n+1}$ a homotopy between $u_0,u_1$ rel $Op(A)$. For any $\varepsilon>0$, there exists an isotropic isotopy $u_t:D^k\rightarrow\mathbb{R}^{2n+1},\,t\in[0,1]$ relative to $Op(A)$ such that $d_{C^0}(\rho_t,u_t)<\varepsilon$ and $\{du_t\}$ is homotopic to $G$ rel $Op(A)$ and $\{0,1\}$.
\end{enumerate}
     
\end{thm}

\subsection{The case of discs}

We prove a contact version of the quantitative $h$-principle for subcritical isotropic discs, using analogous argument as one in the symplectic case provided in \cite{BO21}. Again, by abuse of notation we denote $D^k=[-1,1]^k$.

\begin{prop}\label{h-principleExtended}
    Let $V\subset\mathbb{R}^{2n+1}$ be a bounded open set, $l<k<n$, $I\subset\mathbb{R}$ interval, and $u_0,u_1:D^{l}\times I^{k-l}\rightarrow (V,\xi_{\mathrm{std}})$ subcritical isotropic embeddings which coincide on $Op(\partial D^l\times I^{k-l})$. Assume $u_0$ and $u_1$ are homotopic in $V$ relative to $Op(\partial D^l\times I^{k-l})$, and moreover their differentials $du_0, du_1$ are homotopic in $G^{\mathrm{iso}}(k,2n+1)$ relative to $Op(\partial D^l\times I^{k-l})$ via homotopy $G:[0,1]\times\mathring{D}^l\times I^{k-l}\rightarrow G^{\mathrm{iso}}(k,2n+1)$.
    
    Then there exists a contact isotopy $\psi_t$ with compact support in $V$ such that $\psi_1\circ u_0=u_1$, and $d\psi^t\circ du_0$ is homotopic to $G_t$ rel $Op(\partial D^l\times I^{k-l})$ and $\{0,1\}$.
\end{prop}

\begin{rem}
    Non-relative version holds as well. If $u_0, u_1$ and $du_0,du_1$ are just homotopic (not necessarily relative to $Op(\partial D^l\times I^{k-l})$) we still get contact isotopy $\psi^t$ which achieves $\psi^1\circ u_0=u_1$.
\end{rem}

\begin{proof}
    Let $\rho_t:D^l\times I^{k-l}\rightarrow V$ be a homotopy between $u_0$ and $u_1$ relative to $Op(\partial D^l\times I^{k-l})$, and let $\varepsilon>0$. By Theorem \ref{h-principleDiscs} we get an isotropic isotopy $u_t:D^l\times I^{k-l}\rightarrow V,\,t\in[0,1]$ fixed on $Op(\partial D^l\times I^{k-l})$, such that $d_{C^0}(\rho_t,u_t)<\varepsilon$. Using isotropic isotopy extension theorem for $u_t$ (see Theorem 2.6.2 in \cite{Ge08}) we get a contact isotopy $\psi_t$ generated by a contact Hamiltonian $H_t$ which satisfies $H_t\circ u_t=\alpha\left(\frac{d}{dt}u_t\right)$ and $\psi_t\circ u_0=u_t$, and by cutting off if we can assume that the support of $H_t$ lays inside sufficiently small neighbourhood of the support of the isotopy $u_t$, which is compactly supported in $V$ when $\varepsilon$ is small enough.
\end{proof}

\begin{lemma}\label{lemma2.2}
    Let $A,B\subset D^k$ be closed subsets. Let $u_0,u_1:D^k\hookrightarrow\mathbb{R}^{2n+1}$ be subcritical isotropic embedding that coincide on $Op(A)$. Assume that we are given a homotopy $G_t:D^k\rightarrow G^{\mathrm{iso}}(k,2n+1)$ between $du_0$ and $du_1$ rel $Op(A)$. Let $v_t:D^k\hookrightarrow\mathbb{R}^{2n+1}$ be an isotropic isotopy between $u_0$ and $v_1$ rel $Op(A)$, such that $v_1|_{Op(B)}=u_1|_{Op(B)}$, and such that $\{dv_t|_{Op(B)}\}$ is homotopic to $\{G_t|_{Op(B)}\}$ relative to $Op(A)$ and $\{0,1\}$. Then $dv_1$ and $du_1$ are homotopic rel $Op(A\cup B)$ among maps $D^k\rightarrow G^{\mathrm{iso}}(k,2n+1)$.
\end{lemma}

\begin{proof}
    This is just the contact version of Lemma 2.2 in \cite{BO21}, and the same proof works here.
\end{proof}

\begin{lemma}\label{transvLemma}
    Let $V\subset(\mathbb{R}^{2n+1},\xi_{\mathrm{std}})$ be an open subset. Let $\Sigma_1, \Sigma_2$ be two smooth proper submanifolds of $V$, each of dimension at most $n-1$, which have no intersections in a neighbourhood of $\partial V$. Then there exists an arbitrarily small contact isotopy $(\phi^t)_{t\in[0,1]}$ with compact support in $V$, such that $\phi^1(\Sigma_1)\cap\Sigma_2=\emptyset$.
\end{lemma}

\begin{proof}
    It is enough to show that if $f_1:\Sigma_1\rightarrow V$ and $f_2:\Sigma_2\rightarrow V$ are smooth proper maps such that $f_1(x_1)\neq f_2(x_2)$ for $x_1,x_2$ near $\partial V$, then there exists an arbitrarily small contact isotopy $(\phi^t)_{t\in[0,1]}$, such that $\phi^1\circ f_1(\Sigma_1)\cap f_2(\Sigma_2)=\emptyset$.
    
    Consider $\mathbb{R}^{2n+1}(z,x_1,y_1,\ldots,x_n,y_n)$ with the contact form $\alpha=dz-\sum_{i=1}^ny_idx_i$. Let $K\subset V$ be a compact set such that $f_1(V\setminus K)\cap f_2(V\setminus K)=\emptyset$. Define a smooth map
    \[F:\Sigma_1\times\Sigma_2\rightarrow\mathbb{R}^{2n+1}\quad F(x_1,x_2)=f_1(x_1)-f_2(x_2).\]
    Let $\pi_{xy}:\mathbb{R}^{2n+1}\rightarrow\mathbb{R}^{2n},\,(z,x_1,\ldots,y_n)\mapsto (x_1,\ldots,y_n)$. Since $\mathrm{dim}\,\Sigma_1\times\Sigma_2<2n$ we conclude that $\pi_{xy}\circ F(\Sigma_1\times\Sigma_2)$ has measure $0$ in $\mathbb{R}^{2n}$. This in particular means that we can find arbitrarily small (in norm) vector $v\in\mathbb{R}^{2n}\setminus\pi_{xy}\circ F(\Sigma_1\times\Sigma_2)$. Define contact isotopy
    \[\varphi^t(z,x_1,y_1,\ldots,x_n,y_n)=(0,-tv)+(z-\sum_{i=1}^n tv_{y_i}x_i,x_1,y_1,\ldots,x_n,y_n)\]
    where $v_{y_i}$ is the $y_i$-coordinate of $v$. Note that
    \[\pi_{x,y}(\varphi^1\circ f_1(x_1)-f_2(x_2))=-v+\pi_{x,y}\circ F(x_1,x_2)\neq 0^{2n},\]
    therefore $\varphi^1\circ f_1(\Sigma_1)\cap f_2(\Sigma_2)=\emptyset$.
    
    Let $M=\max_{(z,x_1,\ldots,y_n)\in K}|x_1|+\cdots+|x_n|$. Then $\varphi^1(x)-x=(-\sum_{i=1}^n v_{y_i}x_i,-v)\in \left(||v||\cdot[-M,M]\right)\times\{-v\}$ for $x\in V$, which means that $\varphi^1$ can be arbitrarily close to identity if the norm of $v$ is small enough. Let $f\in C^{\infty}_c(V)$ be a cut-off function such that $f|_K=1$. Define $\phi^t=\phi^t_{fH}$, where $H=\alpha(\frac{d}{dt}\phi^t)$ is a contact Hamiltonian generating $\phi^t$. Finally, we have $\phi^1\circ f_1(\Sigma_1)\cap f_2(\Sigma_2)=\emptyset$ provided norm of $v$ is sufficiently small.
\end{proof}

For the sake of simplicity assume that $D^k=[-1,1]^k$, $D^k(\mu)=[-1-\mu,1+\mu]^k$ and let $u_0,u_1:D^k\rightarrow V\subset\mathbb{R}^{2n+1}$ be subcritical isotropic embeddings. Suppose that $F:D^k\times[0,1]\rightarrow V$ is a homotopy from $u_0$ to $u_1$ of size less than $\varepsilon$. Using Lemma \ref{transvLemma}, we can further assume that the images of $u_0$ and $u_1$ are disjoint. Since $2(k+1)<2n+1$, we can approximate homotopy $F$ by a smooth embedding $\widetilde{F}$ such that the size of $\widetilde{F}$ is less than $\varepsilon$. Additionally, we can use normal neighbourhood of the image of $\widetilde{F}$ in order to extend $\widetilde{F}$ to a smooth embedding (which we again denote by $\widetilde{F}$):
\[\widetilde{F}:D^k(\mu)\times[-\mu,1+\mu]\times[-\mu,\mu]^{2n-k}\hookrightarrow V.\]

Let $\nu=\frac{1}{N}$ for some large $N\in\mathbb{N}$. Consider the grid $\Gamma_0:=(\nu\mathbb{Z})^k\cap D^k$ which generates a cellular decomposition of $D^k$. Denote by $\Gamma_l$ the union of $l$-faces. The set of $k$-faces has a natural integer-valued distance $d_k$, such that $d_k(x,x')$ is the minimal $m$ such that there exists a sequence $x=x_0,x_1,\ldots,x_m=x'$ of $k$-faces such that $x_i\cap x_{i+1}\neq\emptyset$ for each $i\in\{0,\ldots,m-1\}$. Fix $\eta<\nu/2$, and for each $x\in\Gamma_0$ let $U_x$ be the $\eta$-neighbourhood of $\{x\}\times[0,1]\times\{0\}^{2n-k}$. Define $W_x:=\widetilde{F}(U_x)$. For each $k$-face $x_{k}$ let $U_{x_k}$ be the $\eta$-neighbourhood of $x_k\times[0,1]\times\{0\}^{2n-k}$, and set $W_{x_k}:=\widetilde{F}(U_{x_k})$. Additionally, for each $k$-face $x_k$ and each $m\geq 0$, define $W^m_{x_k}:=\bigcup_{d_k(x_k,x_k')\leq m}W_{x_k'}$ (note that $W^m_{x_k}$ is a topological ball). Finally, define $W:=\widetilde{F}(U)$, where $U$ is the $\eta$-neighbourhood of $D^k\times[0,1]\times\{0\}^{2n-k}$.

\begin{prop}\label{prop3.3}
    There exists a sequence of Hamiltonian isotopies $\{\Psi^t_l\}_{l=0}^k$ supported inside $W$, and embeddings $v_0:=\Psi^1_0\circ u_0,\,v_{l}=\Psi^1_l\circ v_{l-1}$ for $l\in\{1,\ldots,k\}$ such that
    \begin{enumerate}[label=(\Roman*)]
        \item $v_l|_{Op(\Gamma_l)}=u_1|_{Op(\Gamma_l)}$ for each $l\in\{0,\ldots,k\}$,
        \item $v_l(x_k)\subset W^{3^l-1}_{x_k}$ for each $k$-face $x_k$ and $l\in\{0,\ldots,k-1\}$,
        \item $\Psi^t_0(W_{x_k})\subset W_{x_k}$, $\Psi^t_l(W_{x_k})\subset W^{2\cdot 3^l-1}_{x_k},\,l\in\{1,\ldots,k\}$ for each $k$-face $x_k$,
        \item $v_l(\mathring{x}_{l+1})\cap u_1(\mathring{x}_{l+1}')=\emptyset$ for pairs of distinct $(l+1)$-faces $x_{l+1}$ and $x_{l+1}'$ and $l\in\{0,\ldots,k-1\}$,
        \item $dv_l$ and $du_1$ are homotopic $rel$ $Op(\Gamma_l)$ inside $G^{\mathrm{iso}}(k,2n+1)$, for $l\in\{0,\ldots,k-1\}$.
    \end{enumerate}
\end{prop}

\begin{proof}
    Since disc is contractible, there exists a homotopy $G_t:D^k\rightarrow G^{\mathrm{iso}}(k,2n+1)$ between $du_0$ and $du_1$.\\

    We begin with the $0$-skeleton. Pick $\rho<\eta$ and for each $x\in\Gamma_0$ let $D_{\rho}(x)$ be a $\rho$-neighbourhood of $x$ in $D^k(\mu)$. Note that $\widetilde{F}$ provides an isotopy between ${u_0}|_{D_{\rho}(x)}$ and ${u_1}|_{D_{\rho}(x)}$ supported in $W_x$, therefore we can use Proposition \ref{h-principleExtended} (at this point we do not need the relative version) to get a Hamiltonian isotopy $\psi^t_{x}$ supported in $W_x$ such that $\psi^1_x\circ {u_0}|_{D_{\rho}(x)}={u_1}|_{D_{\rho}(x)}$, and $d\psi^t_x\circ du_0|_{D_{\rho}(x)}$ is homotopic to $G_t$ rel $\{0,1\}$.
    Define $\psi^t_0:=\circ\psi^t_{x}$ where the composition runs over all $x\in\Gamma_0$. 
    Then $d\psi^t_0\circ du_0|_{Op(\Gamma_0)}$ is homotopic to $G_t$ rel $\{0,1\}$, and first three properties obviously hold, so the only potential problem might be the fourth property. However, since $\psi^1_0\circ u_0=u_1$ in the neighbourhood of each point $x_0\in\Gamma_0$, we can find small $r>0$ such that the property $(IV)$ holds in all closed balls $\overline{B(x_0,r)}\subset W_{x_0}$. Now we can apply Lemma \ref{transvLemma} to get an arbitrarily small Hamiltonian perturbation $\widetilde{\psi}^t$ supported in $\bigcup_{x_0\in\Gamma_0}W_{x_0}\setminus\overline{B(x_0,r)}$ that will achieve $\widetilde{\psi}^1\circ\psi_0^1\circ u_0(x_1)\cap u_1(x_1')=\emptyset$ for every pair of different 1-faces $x_1,x_1'$. Now $\Psi^t_0:=\widetilde{\psi}^t\star\psi^t_0$, verifies property $(IV)$. Since $\widetilde{\psi}^t=\mathrm{Id}$ in $Op(\Gamma_0)$, we still have $d\Psi^t_0\circ du_0|_{Op(\Gamma_0)}$ is homotopic to $G_t$ rel $\{0,1\}$, and the first three properties still remain true if the perturbation is small enough.\\
    
    Assume that the sequence $\{\Psi^t_i\}_{i=0}^{l-1}$ has been constructed, and we proceed with the induction step. Recall that $v_{l-1}|_{Op(\Gamma_{l-1})}=u_1|_{Op(\Gamma_{l-1})}$, $v_{l-1}(x_k)\subset W^{3^{l-1}-1}_{x_k}$ for each $k$-face $x_k$, and we have homotopy $G^l_t:D^k\rightarrow G^{\mathrm{iso}}(k,2n+1)$ between $dv_{l-1}$ and $du_1$ rel $Op(\Gamma_{l-1})$. Fix an $l$-face $x_l$, and let $x_k$ be the $k$-face which contains $x_l$. Let $\hat{x}_l\Subset\mathring{x}_l$ be a closed box such that $v_{l-1}|_{Op(x_l\setminus\hat{x}_l)}=u_1|_{Op(x_l\setminus\hat{x}_l)}$. Note that $u_{l-1}(\hat{x}_l)$ and $u_{1}(\hat{x}_l)$ both lay inside topologigal ball $W^{3^{l-1}-1}_{x_k}$, and coincide near boundary, hence there exists a homotopy
    \[\sigma_{x_l}:\hat{x}_l\times[0,1]\rightarrow W^{3^{l-1}-1}_{x_k}\]
    such that $\sigma_{x_l}(\cdot,0)=v_{l-1}|_{\hat{x}_l},\,\sigma_{x_l}(\cdot,1)=u_1|_{\hat{x}_l}$ and $\sigma_{x_l}|_{Op(\partial\hat{x_l}\times[0,1])}=u_1|_{Op(\partial\hat{x_l}\times[0,1])}$. Since $l<n$, we can use property $(IV)$ and general position argument to ensure that moreover $\mathrm{Im}\,\sigma_{x_l}$ admits a regular neighbourhood $\mathcal{V}_{x_l}\subset W^{3^{l-1}-1}_{x_k}$ such that $\mathcal{V}_{x_l}\cap\mathcal{V}_{x_l'}=\emptyset$ for each pair of distinct $l$-faces $x_l,x_l'$.\\
    
    Now we can use Proposition \ref{h-principleExtended} to get a Hamiltonian isotopy $\psi^t_{x_l}$ supported inside $\mathcal{V}_{x_l}$, such that $\psi^1_{x_l}\circ v_{l-1}|_{Op(\hat{x}_l)}=u_1|_{Op(\hat{x}_l)}$, and the restriction $d(\psi^t_{x_l}\circ v_{l-1})|_{Op(\partial\hat{x}_l)}$ is homotopic to $G^l_t$ relative to $Op(\partial\hat{x}_l)$ and $\{0,1\}$. Define $\psi^t_l:=\circ\psi^t_{x_l}$ where the composition runs over all $l$-faces $x_l$. We also define $\hat{v}_l:=\psi^1_l\circ v_{l-1}$. Since all $\mathcal{V}_{x_l}$ are pairwise disjoint, we get $\hat{v}_l|_{Op(x_l)}=u_1|_{Op(x_l)}$ for each $l$-face $x_l$, i.e. $\hat{v}_l$ and $u_1$ coincide on $Op(\Gamma_l)$ which verifies property $(I)$. By Lemma \ref{lemma2.2}, $\hat{v}_l$ satisfies property $(V)$.\\
    
    Let $x$ be any $k$-face, and let $x_l$ be the $l$-face such that $\mathcal{V}_{x_l}\cap W_x\neq\emptyset$. Let $x_k$ be a $k$-face containing $x_l$ as above, so that $\mathcal{V}_{x_l}\subset W^{3^{l-1}-1}_{x_k}$. Then the distance $d_k(x,x_k)$ is at most $3^{l-1}$, therefore we have $\mathcal{V}_{x_l}\subset W^{3^{l-1}-1}_{x_k}\subset W_x^{2\cdot 3^{l-1}-1}$. As a result (since $\psi^t_l$ is supported in disjoint union of sets $\mathcal{V}_{x_l}$) we get
    \begin{equation}\label{eq.subset}
        \psi^t_l(W_x)\subset W^{2\cdot 3^{l-1}-1}.
    \end{equation}
    It can happen that $\hat{v}_l$ fails to satisfy property $(IV)$, i.e. there might be two distinct $(l+1)$-faces $x_{l+1},x_{l+1}'$ such that $\hat{v}_l(\mathring{x}_{l+1})\cap u_1(\mathring{x}_{l+1}')\neq\emptyset$. However, we have $\hat{v}_l|_{Op(\Gamma_l)}=u_1|_{Op(\Gamma_l)}$, thus the set $\hat{v}_l(\mathring{x}_{l+1})\cap u_1(\mathring{x}_{l+1}')$ is compactly contained in $W\setminus u_1(\Gamma_l)$. By the Lemma \ref{transvLemma} we can find arbitrarily small Hamiltonian isotopy $\varphi^t_l$, compactly supported in $W\setminus\Gamma_l$, such that $v_l:=\varphi^1_l\circ\hat{v}_l$ verifies property $(IV)$. By the smallness of the flow $\varphi^t_l$ and by (\ref{eq.subset}), the flow $\Psi^t_l:=\varphi^t_l\star\psi^t_l$ satisfies $\Psi^t_l(W_x)\subset W^{2\cdot 3^l-1}_x$ for any $k$-face $x$, and thus property $(III)$ holds. Support of $\varphi^t_l$ is compactly supported in $W\setminus\Gamma_l$, so properties $(I)$ and $(V)$ still remain valid. Finally, property $(II)$ follows as well: if $x$ is any $k$-face, then by assumption, $v_{l-1}(x)\subset W^{3^{l-1}-1}_x$, so we get
    \[v_l(x)=\Psi^1_l\circ v_{l-1}(x)\subset\Psi^1_l\left(W^{3^{l-1}-1}_x\right)=\bigcup_{d_k(x,y)\leq 3^{l-1}-1}\Psi^1_l(W_y)\]
    \[\subset\bigcup_{d(x,y)\leq 3^{l-1}-1}W^{2\cdot 3^{l-1}}_y=W_x^{3^{l-1}-1+2\cdot 3^{l-1}}=W^{3^l-1}_x.\]
\end{proof}

Proposition \ref{prop3.3} immediately implies Theorem \ref{quant-h-princ-discs} for discs. Indeed, we can denote by $(\Psi^t)_{t\in[0,1]}$ reparametrized concatenation of the flows $\Psi^t_1,\Psi^t_2,\ldots,\Psi^t_k$, and by the property $(III)$ we conclude that for each $k$-face $x$ we have $\Psi^t(W_x)\subset W_x^{3^{k+1}-3}$ because $\sum_{i=1}^k2\cdot 3^i=3^{k+1}-3$. The flow $\Psi^t$ is supported in $W=\bigcup_{x\in\Gamma_k}W_x\subset V$, and if the step $\nu$ of the grid is chosen to be small enough, then for each $k$-face $x$, the diameter of $W^{3^{k+1}-1}$ is less than $\varepsilon$. This means that the size of the flow $\Psi^t$ is less than $\varepsilon$, and by the property $(I)$ we have $\Psi^1\circ u_0=u_1$ on $D^k$.\Qed

\subsection{The case of $S^1$}

We prove a contact version of the quantitative $h$-principle for subcritical isotropic embeddings of $S^1$ (Theorem \ref{Theorem1} (b)). Let $F:S^1\times[0,1]\rightarrow V$ be a homotopy between $\gamma_0$ and $\gamma_1$ of size less than $\varepsilon$. Analogously to the case of discs, we can assume that $F$ is an embedding, and we can extend $F$ to an embedding (which we again denote by $F$):
\[F:S^1\times[-\mu,1+\mu]\times[-\mu,\mu]^{2n-1}\hookrightarrow V,\]
for some small $\mu>0$. Let $\eta=\frac{1}{N}>2\delta>0,\,N\in\mathbb{N}$, and consider grid $\Gamma_0=(\mathbb{R}\cap\eta\mathbb{Z})/\mathbb{Z}\subset S^1=\mathbb{R}/\mathbb{Z}$. For each $x\in \Gamma_0$ let $U_x$ be a $\delta$-neighbourhood of $\{x\}\times[0,1]\times\{0\}^{2n-1}$ and define $W_x=F(U_x)$. Similarly, for each $1$-cell $I=[x_i,x_{i+1}]\subset S^1$, let $U_I$ be a $\delta$-neighbourhood of $I\times[0,1]\times\{0\}^{2n-1}$ and define $W_I=F(U_I)$.\\

Since $G^{\mathrm{iso}}(1,2n+1)$ is isomorphic to $\mathbb{R}^{2n}\setminus\{0\}$, we conclude that there exists a homotopy $G_t:S^1\rightarrow G^{\mathrm{iso}}(1,2n+1)$ between $d\gamma_0$ and $d\gamma_1$.\\

For each $x\in\Gamma_0$ define $I(x)=[x-\rho,x+\rho]\subset S^1$ (where $\rho<\delta$). Then $F$ gives isotopy between $\gamma_0|_{I(x)}$ and $\gamma_1|_{I(x)}$ inside $W_x$. Using $h$-principle for isotropic discs (Theorem \ref{h-principleExtended}) we get contact isotopy $\psi^t_x$ supported inside $W_x$ such that $\psi^1_x\circ \gamma_0|_{I(x)}=\gamma_1|_{I(x)}$ and $d\psi^t_x\circ d\gamma_0|_{I(x)}$ is homotopic to $G_t$ rel $\{0,1\}$.

Define $\psi_0^t:=\circ\psi^t_x$, where composition runs over all $x\in\Gamma_0$. Then $\psi_0^t\circ \gamma_0|_{Op(\Gamma_0)}=\gamma_1|_{Op(\Gamma_0)}$ and $d\psi^t_0\circ d\gamma_0|_{I(x)}$ is homotopic to $G_t$ rel $\{0,1\}$. Using Lemma \ref{transvLemma}, let $\widetilde{\phi}^t$ be contact isotopy which achieves 
\[\widetilde{\phi}^1\circ\psi^1_0\circ\gamma_0(\mathring{I})\cap\gamma_1(\mathring{I'})=\emptyset,\]
for every pair of distinct 1-cells $I,I'$. Define
\[\Psi_0^t:=\widetilde{\phi}^t\star\psi^t_0,\quad v_0:=\Psi_0^1\circ \gamma_0.\]
Note that $\widetilde{G}_t:=d\Psi^t_0\circ d\gamma_0|_{I(x)}$ is still homotopic to $G_t$ rel $\{0,1\}$, and for each 1-cell $I$ we have $v_0(I),u_1(I)\subset W_I$ and $v_0|_{Op(\partial I)}=\gamma_1|_{Op(\partial I)}$.\\
    
Consider now slightly smaller interval $\widehat{I}\Subset I\subset S^1$ and pick any homotopy
\[\sigma_{I}:\widehat{I}\times[0,1]\rightarrow W_I,\quad\sigma_I(\cdot,0)=v_0,\,\sigma_I(\cdot,1)=u_1.\]
General position argument implies that we can moreover assume that the image $\mathrm{Im}\,\sigma_I$ admits a regular neighbourhood $\mathcal{V}_{I}\subset W_I$, such that $\mathcal{V}_{I}\cap\mathcal{V}_{I'}=\emptyset$ for each pair of distinct 1-cells $I,I'$. Now we are ready to apply Theorem \ref{h-principleExtended} which gives us contact isotopies $\psi^t_I$ supported in $\mathcal{V}_I$, such that
$\psi^1_I\circ v_0|_{Op(\widehat{I})}=u_1|_{Op(\widehat{I})}$ and the restriction $d(\psi^t_I\circ v_0)|_{Op(\partial\widehat{I})}$ is homotopic to $\widetilde{G}_t$ relative to $Op(\partial\widehat{I})$ and $\{0,1\}$. Let $\Psi_1^t=\circ\psi^t_I$, where composition runs over all $1$-faces $I$. Finally we define
$\Psi^t=\Psi_1^t\star\Psi_0^t$. Note that $\Psi^1\circ\gamma_0=\gamma_1$, and assuming $\eta>0$ is small enough, size of $\Psi^t$ will be less than $\varepsilon$.\Qed

\section{Taking subcritical isotropic curve to a transverse one}
\addtocontents{toc}{\protect\setcounter{tocdepth}{2}}
\subsection{Open case}\label{sect4.2}

We prove Theorem \ref{Theorem2} (a). More precisely, we pick any transverse curve in $(V,\xi)$ and provide an explicit construction of a contact homeomorphism, supported in tubular neighbourhood of a given transverse curve, which maps some isotropic curve to the transverse one.\\

By the contact neighborhood theorem, a sufficiently small neighbourhood of a transverse curve can be contactly embedded to an open neighbourhood of $[0,1]\times\{0\}^{2n}\subset\mathbb{R}^{2n+1}$ with the contact form $dz+\sum_{j=1}^n(x_jdy_j-y_jdx_j)$. From now on we work with the transverse curve \[v:[0,1]\rightarrow\mathbb{R}^{2n+1},\quad t\mapsto(t,0,0,\ldots,0).\]

Let $x_1=\sqrt{\rho}\sin\varphi$ and $y_1=\sqrt{\rho}\cos\varphi$. Then the contact form becomes
\[dz-\rho d\varphi+\sum_{j=2}^n(x_jdy_j-y_jdx_j),\]
and from now on we will use coordinates $(z,\varphi,\rho,x_2,y_2,\ldots,x_n,y_n)$. Define a front projection as $\pi_{F}(z,\varphi,\rho,x_2,y_2,\ldots,x_n,y_n)=(z,\varphi)$.\\

\begin{defn}
    We call a curve $t\mapsto (z(t),\varphi(t)),\,t\in[0,1]$ \textbf{admissible}, if it is immersed everywhere except at finitely many singular points, and the following conditions hold
    \begin{itemize}
      \item $\mathrm{sgn}(z'(t))=\mathrm{sgn}(\varphi'(t))$ for all $t\in [0,1]$,
      \item all self intersections are transverse,
      \item function $t\mapsto\lim_{\tau\rightarrow t}\frac{z'(\tau)}{\varphi'(\tau)}$ is smooth.
    \end{itemize}
\end{defn}

Note that any admissible curve in the front projection lifts to an isotropic embedding. For instance any smooth function $f$ gives rise to an isotropic embedding
\[t\mapsto (z(t),\varphi(t),\lim_{\tau\rightarrow t}\frac{z'(\tau)}{\varphi'(\tau)},f(t),f(t),\ldots,f(t)).\]

We are going to construct a contact homeomorphism $h$ supported inside tubular neighborhood of a transverse curve $v$, which maps isotropic curve $\gamma(t)=(t,ct,1/c,0,\ldots,0)$ to $v$ (where $c>0$ is small enough).

\begin{lemma}\label{stechingLem}
    Let $\gamma(t)=(t,ct,1/c,0,0,\ldots,0)\,, c>0$ be an isotropic embedding. Let $U$ be an open neighbourhood of $\gamma([0,1])$, and let $\Sigma:[0,1]^2\rightarrow\mathbb{R}^{2n+1},\,(s,t)\mapsto(t,ct,(1-s)/c,0,0,\ldots,0)$\footnote{Local coordinates are not defined for $\rho=0$, but in this case we mean $\Sigma(1,t)=v(t)$} be an embedded surface. Then there exists a contact diffeomorphism $\psi$ supported inside $Op(\Sigma)$ which satisfies $\Sigma([0,1]^2)\subset\psi(U),\,\psi\circ\gamma=\gamma$, and $d_{C^0}(\psi,\mathrm{Id})<\frac{3}{c}$.
\end{lemma}

\begin{proof}
    We will construct $\phi$ using a Hamiltonian flow which preserves the surface $\Sigma$. More precisely, for any smooth function $\lambda:[0,1]\rightarrow[0,1]$, we will find a contact Hamiltonian $H_{\lambda}$, such that its flow $\phi_{H_{\lambda}}^{\tau}$ satisfies
     \[\phi_{H_{\lambda}}^{\tau}\left(\Sigma(s,t)\right)=\Sigma(se^{\lambda(t)\tau},t).\]
    For this it is enough to find a Hamiltonian which satisfies $X_{H_{\lambda}}\circ\Sigma(s,t)=s\lambda(t)\cdot\frac{\partial}{\partial s}\Sigma(s,t)$, which is the case if we define $H_{\lambda}$ as
    \[ H_{\lambda}=(z-\varphi/c)\cdot\lambda\left(\frac{c\varphi+z}{c^2+1}\right),\]
    inside sufficiently small neighbourhood of $\Sigma$ and $0$ outside slightly larger neighbourhood. Consider now $1-\lambda(t)$ instead of $\lambda(t)$ to get $\phi_{H_{1-\lambda}}^{\tau}(\Sigma(s,t))=\Sigma(s,e^{(1-\lambda(t))\tau},t)$. Finally, we define a contact diffeomoprhism $\psi=\phi_{H_{1-\lambda}}^{\tau}\circ\phi_{H_{\lambda}}^{\tau}$ which satisfies
    \[\psi(\Sigma(s,t))=\Sigma(se^{\tau},t).\]
    For $s_0$ small enough, $\Sigma([0,s_0]\times [0,1])\subset U$, thus for $\tau=-\log(s_0)$ we have
    \[\psi\left(\Sigma([0,s_0]\times [0,1])\right)=\Sigma([0,1]\times[0,1]),\]
    and it follows that property $\Sigma([0,1]^2)\subset\psi(U)$. Since $X_H(x)=0$ for $x\in\gamma([0,1])$ we will have $\psi\circ\gamma=\gamma$. It only remains to show that $\psi$ is $C^0$-close to the identity. Let $0=t_0<t_1<\ldots<t_{2N}=1$ such that $t_{j+1}-t_j=1/(2N)$ and $N\gg 1$. Choose $\lambda$ so that $\mathrm{supp}(\lambda)\subset\bigsqcup_{j=0}^{N-1}[t_{2j}-\varepsilon,t_{2j+1}+\varepsilon]$ and $\mathrm{supp}(1-\lambda)\subset\bigsqcup_{j=1}^{N}[t_{2j-1}-\varepsilon,t_{2j}+\varepsilon]$, where $\varepsilon=1/N^2$. This means that \[\mathrm{supp}(\phi^{H_{\lambda}}_{\tau}|_{\Sigma})\subset\bigsqcup_{j=0}^{N-1}\Sigma\left([0,1]\times[t_{2j}-\varepsilon,t_{2j+1}+\varepsilon]\right),\]
    but the diameter of $\Sigma\left([0,1]\times[t_{2j}-\varepsilon,t_{2j+1}+\varepsilon]\right)$ approaches to $\frac{1}{c}$ as $N\rightarrow\infty$. This means that the support of $\phi_{H_{\lambda}}^{\tau}$ will be disjoint union of sets of diameters less than $\frac{3}{2c}$, thus $d_{C^0}(\phi_{H_{\lambda}}^{\tau},\mathrm{Id})<\frac{3}{2c}$. Similarly $d_{C^0}(\phi_{H_{1-\lambda}}^{\tau},\mathrm{Id})<\frac{3}{2c}$, so using the triangle inequality we finish the proof.
\end{proof}

Let $\delta_1>\varepsilon_1>0$ and $\gamma_1(t)=(t,\varepsilon_1t,\frac{1}{\varepsilon_1},0,\ldots,0)$. Let $\{U_i\}_{i\geq 1}$ and $\{W_i\}_{i\geq 1}$ be decreasing sequences of open sets in $\mathbb{R}^{2n+1}$ such that $\bigcap_{i\geq 1}U_i=\gamma_1([0,1])$ and $\bigcap_{i\geq 1}W_i=v([0,1])$. We will inductively construct
\begin{itemize}
    \item decreasing sequences $\{\delta_i\}_{i\geq 1}$ and $\{\varepsilon_i\}_{i\geq 1}$ with $0<\varepsilon_i<\delta_i<\frac{1}{2^i}$,
    \item sequence of contact diffeomorphisms $\{\psi_i\}_{i\geq 1}$,
\end{itemize}
such that the following conditions hold
\begin{enumerate}[label={($\mathcal{I}$\arabic*)}]
    \item $\psi_i$ is supported inside $Op([0,1])\times B^{2n}(\delta_m)\subset\varphi_{i-1}(U_i)\cap W_i$, where we define $\varphi_i$ as $\varphi_i:=\psi_i\circ\psi_{i-1}\circ\cdots\circ\psi_1$,
    \item $\psi_i\circ\gamma_i=\gamma_{i+1}$, where $\gamma_{i}=(t,\varepsilon_it,1/\varepsilon_i,0,\ldots,0)$,
    \item $d_{C^0}(\psi_i,\mathrm{Id})<15\varepsilon_i$,
    \item $\Sigma_i([0,1]^2)\subset\varphi_{i-1}(U_i)\cap W_i$, where $\Sigma_i(s,t)=(t,\varepsilon_it,(1-s)/\varepsilon_i,0,\ldots,0)$.
\end{enumerate}

Suppose that we have constructed contact diffeomorphisms $\psi_1,\ldots,\psi_{m-1}$. We apply the Lemma \ref{stechingLem} for $\gamma_{m-1}$ and open set $\varphi_{m-1}(U_{m+1})$ to get contact diffeomorphism $\psi'_m$ supported in $Op(\Sigma_m([0,1]^2))\subset Op([0,1])\times B^{2n}(\delta_m)$ and $d_{C^0}(\psi'_m,\mathrm{Id})<3\varepsilon_m<3\delta_m$, such that $\Sigma_m([0,1]^2)\subset\psi'_m\circ\varphi_{m-1}(U_{m+1})$. Now we pick $\varepsilon_{m+1}<\delta_{m+1}<1/2^{m+1}$ small enough, such that $Op([0,1])\times B(\delta_{m+1})\subset\psi'_m\circ\varphi_{m-1}(U_{m+1})\cap W_{m+1}$. Moreover, we assume that $\delta_{m+1}$-neighbourhood of $\Sigma_m([0,1]^2)$ lays inside $\psi'_m\circ\varphi_{m-1}(U_{m+1})\cap Op([0,1])\times B(\delta_m)$.

\begin{figure}[h]
    \centering
    \includegraphics[width=\textwidth]{ 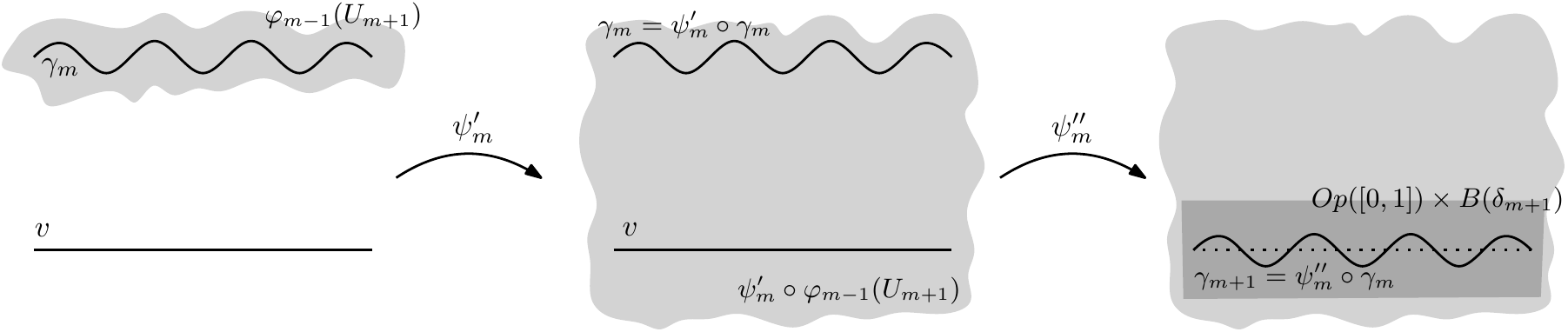}
    \label{Isotopy}
\end{figure}

\begin{prop}\label{InductionStep}
    Let $\gamma_{m+1}=(t,\varepsilon_{m+1}t,1/\varepsilon_{m+1},0,\ldots,0)$.  Then there exists Hamiltonian diffeomorphism $\psi''_m$ which satisfies
    \begin{enumerate}
        \item $\psi''_m\circ\gamma_m=\gamma_{m+1}$,
        \item support of $\psi''_m$ is a subset of $\psi'_m\circ\varphi_{m-1}(U_{m+1})\cap (Op([0,1])\times B(\delta_m))$,
        \item $d_{C^0}(\psi''_m,\mathrm{Id})<12\varepsilon_m$.
    \end{enumerate}
\end{prop}

Let $\psi_m=\psi''_m\circ\psi'_m$. It is straightforward to check that $\psi_m$ satisfies all 4 conditions of the induction hypothesis.\\

As a result of induction, we get a sequence of contact diffeomorphisms $\{\varphi_i\}_{i\geq 1}$ which satisfy properties that we need in order to apply the Lemma \ref{homeomorphismLemma} which tells us that the sequence $C^0$-converges to a homeomorphism $h$ which satisfies $h\circ\gamma_1=v$.\Qed

\subsubsection{Proof of the proposition \ref{InductionStep}}

We will get $\psi''_m$ as a composition of contact diffeomorphisms obtained by extending isotropic isotopies. In order to ensure that $\psi''_m$ stays $C^0$-close to the identity, we will only consider isotopies which have supports consisting of the disjoint union of sets with diameters less than $\delta_{m}$.\\

Let $V\subset\psi'_m\circ\varphi_{m-1}(U_{m+1})$ be the $\delta_{m+1}$-neighbourhood of the image $\Sigma_m([0,1]^2)$. To ensure that the support of $\gamma_s$ lays is inside $V$, we will only consider isotopies of the form $\gamma_s(t)=(z_s(t),\varphi_s(t),\rho_s(t),f_s(t),\ldots,f_s(t))$ where $f_s\in C^{\infty}(\mathbb{R})$ and the following conditions are satisfied
\begin{itemize}
    \item $|f_s|<\delta_{m+1}$, $\rho_s<\delta_m$, $z_s(t)\in(-\delta_{m+1},1+\delta_{m+1})$,
    \item $z_s(t)-\rho_s(t)\varphi_s(t)=0$ for all $t\in[0,1]$ meaning that the isotopy is isotropic,
    \item at least one of the following two conditions is satisfied:
    \begin{enumerate}[label={(\arabic*)}]
        \item $(z_s(t),\varphi_s(t))\in\pi_F\circ V=\bigcup_{t\in[0,1]}\{(z,\varphi)\mid |z-t|<\delta_{m+1},|\varphi-t/\varepsilon_m|<\delta_{m+1}\}$,
        \item $\rho_s(t)<\delta_{m+1}$ meaning that $\gamma_s(t)\in Op([0,1])\times B(\delta_{m+1})$.
    \end{enumerate}
\end{itemize}

In the following steps we assume that all isotopies satisfy above conditions, which implies that the support of $\psi''_m$ is contained inside $V$ and therefore property 2. in the proposition \ref{InductionStep} will be satisfied. As opposed to the 3-dimensional case, we can always locally perform "stabilization" in the front projection of the subcritical isotropic curve.

\begin{claim}\label{Cusp}
           Let $\gamma_0(t)=(z_0(t),\varphi_0(t),\rho_0(t),0,0,\ldots,0)$ be an isotropic embedding, and let $z_0'(t_0)\neq 0$. For each $\varepsilon>0$ there exists an isotropic isotopy $\Gamma:[0,1]^2\rightarrow\mathbb{R}^{2n+1}$ supported in $\varepsilon$-neighbourhood of $\gamma(t_0)$, such that $\Gamma(0,\cdot)=\gamma_0$ and $\Gamma(1,\cdot)=\gamma_1$, where $\gamma_1(t)=(z_1(t),\varphi_1(t),\rho_1(t),0,\ldots,0)$ with the front projection as on the image below.
\end{claim}

\begin{figure}[h]
    \centering
    \includegraphics[scale=0.8]{ 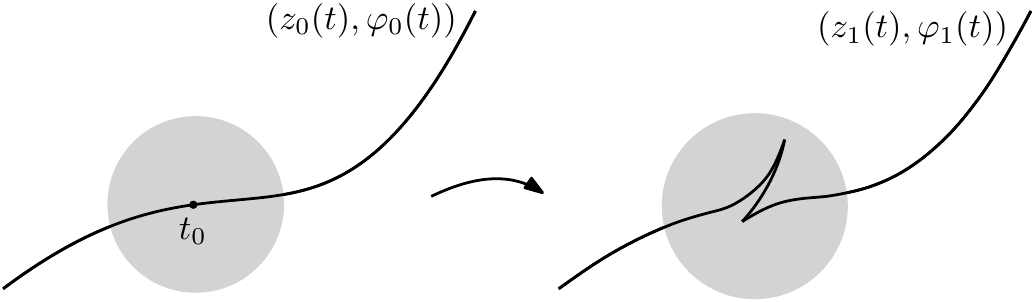}
    \caption{Curve on the right have two cusps and coincides with the curve on the left outside gray neighbourhood}
    \label{CreatingCusp}
\end{figure}

\begin{proof}
    Assume $z'(t_0)>0$ and let $I$ be a sufficiently small interval around $t_0$ such that the restriction $z'|_I>0$ and $\gamma(I)\subset B(\gamma(t_0),\varepsilon/2)$. We can choose $\gamma_1$ so that $\gamma_0(t)=\gamma_1(t)$ for $t\in[0,1]\setminus I$, $|I|<\frac{\varepsilon}{2\max_{t\in[0,1]}\{|\varphi_0|,|\varphi_1(t)|\}}$, and $\frac{1}{2}\leq\frac{\rho_0}{\rho_1}\leq 2$.
    Let $f:[0,1]\rightarrow\mathbb{R}$ be a smooth function which satisfies $|f|<\varepsilon$, $f'(t)>0$ for $t\in I$, and $f(t)=0$ if $\gamma(t)\not\in B(\gamma(t_0),\varepsilon)$.
    
    Define an isotopy of isotropic embeddings
    \[\Gamma(s,t)=((1-s)z_0(t)+sz_1(t),\varphi_0(t_0)+\int_{t_0}^t\frac{(1-s)z'_0(\tau)+sz'_1(\tau)}{(1-s)\rho_0(\tau)+s\rho_1(\tau)}d\tau,\]
    \[(1-s)\rho_0(t)+s\rho_1(t),\sin(s\pi)f(t),\ldots,\sin(s\pi)f(t)),\quad s\in[0,1].\]
    Note that $\Gamma(0,\cdot)=\gamma_0$, $\Gamma(1,\cdot)=\gamma_1$ and $\Gamma(s,t)=\gamma_1(t)$ for $(s,t)\in[0,1]\times([0,1]\setminus I)$. Let us verify that $\varphi$-coordinate stays in $\varepsilon$-neighbourhood of $\varphi_0(t_0)$. It is enough to check that for $t\in I$ we have
    \[|\varphi_s(t)-\varphi_0(t)|=\Big\lvert\int_{t_0}^t \frac{(1-s)z_0'(\tau)+sz_1'(\tau)}{(1-s)\rho_0(\tau)+s\rho_1(\tau)}d\tau\Big\rvert=|t-t_0|\cdot\Big\lvert\frac{(1-s)z_0'(\xi)+sz_1'(\xi)}{(1-s)\rho_0(\xi)+s\rho_1(\xi)}\Big\rvert=\]
    \[|t-t_0|\cdot\Big\lvert\frac{(1-s)\varphi_0(\xi)}{(1-s)+s\rho_1(\xi)/\rho_0(\xi)}+\frac{s\varphi_1(\xi)}{s+(1-s)\rho_0(\xi)/\rho_1(\xi)}\Big\rvert\leq 2|I|\cdot\max_{t\in[0,1]}\{|\varphi_0(t)|,|\varphi_1(t)|\}<\varepsilon.\]
\end{proof}

Let $\varepsilon>0$ be small enough and let $0=t_0<t_1<\ldots<t_{N}=1$ such that $t_{i+1}-t_i<\varepsilon$. Additionally, for each $0\leq i<N$ we pick real numbers $p_i,q_i$ such that $t_i<p_i<q_i<t_{i+1}$.\\

\textbf{Step I (creating cusps):} Using Claim \ref{Cusp} we create two cusps (at $p_i$ and $q_i$) inside each of the intervals $(t_i,t_{i+1})$. We end up with the curve $(z_1(t),\varphi_1(t),\rho_1(t),0,\ldots,0)$ such that $\frac{z_1(p_{i+1})-z_1(q_i)}{\varphi_1(p_{i+1})-\varphi_1(q_i)}<\delta_{m+1}$ (we can do this if $\varepsilon$ is small enough) and $\rho_1|_{\sqcup_{i=0}^{N-1}[p_i,q_i]}<\delta_{m+1}$. This isotopy is supported inside $\bigsqcup_{i=0}^{N-1}\pi_z^{-1}(Op([p_i,q_i]))\cap V$, and therefore we can extend it to a contact diffeomorphism $\Psi_1$ such that $d_{C^0}(\Psi_1,\mathrm{Id})<\delta_m$ (see figure \ref{StepI}).

\begin{figure}[h]
    \centering
    \includegraphics[width=\textwidth]{ 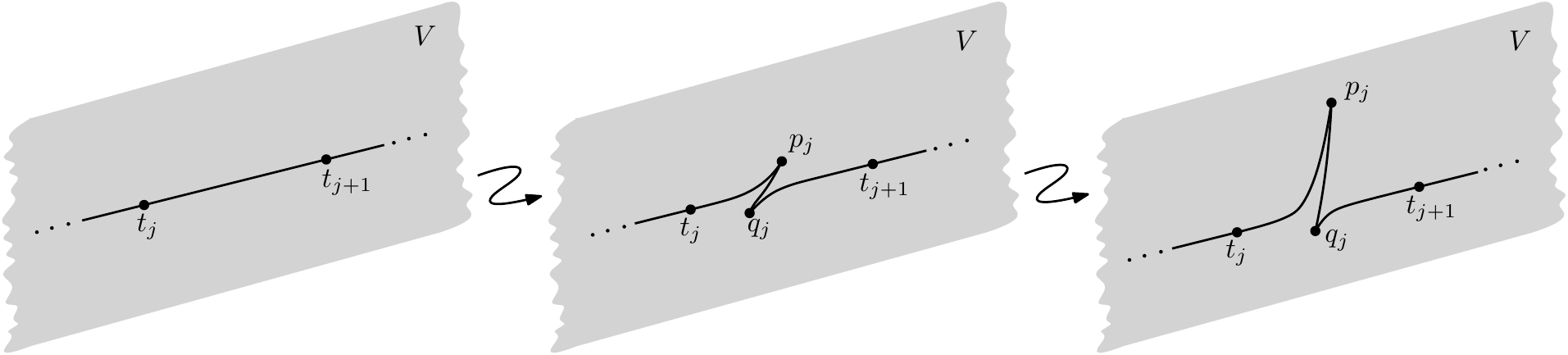}
    \caption{We create two cusps and then lift up the upper cusp, while keeping the $z$ coordinate fixed}
    \label{StepI}
\end{figure}

\textbf{Step II:} Let $g\in C^{\infty}(\mathbb{R})$ such that $g'>0$, $g(0)=0$ and $g(1)<\delta_{m+1}$. Denote $G(t)=(g(t),\ldots,g(t))\in\mathbb{R}^{2n-1}$. Let $\mu:[0,1]\rightarrow[0,1]$ such that $\mu(t)=1$ for $t\in[p_{2i},q_{2i}]$ and $\mu(t)=0$ for $t\in[p_{2i+1},q_{2i+1}]$. Define isotropic isotopy $(z_1(t),\varphi_1(t),\rho_1(t),s\lambda(t)G(t))$ which is supported inside $\bigsqcup_{i=0}^{N-1}\pi_z^{-1}((q_{2i-1},p_{2i+1}))\cap V$ and therefore extends to a contact diffeomorphism $\Psi'_2$ such that $d_{C^0}(\Psi'_2,\mathrm{Id})<\delta_m$. Similarly, isotopy $(z_1(t),\varphi_1(t),\rho_1(t),(s+\lambda(t)(1-s))G(t))$ extends to a contact diffeomorphism $\Psi''_2$ which satisfies $d_{C^0}(\Psi''_2,\mathrm{Id})<\delta_m$. We end up with the isotropic curve $(z_1(t),\varphi_1(t),\rho_1(t),G(t))$. 

\begin{figure}[h]
    \centering
    \includegraphics[scale=0.8]{ 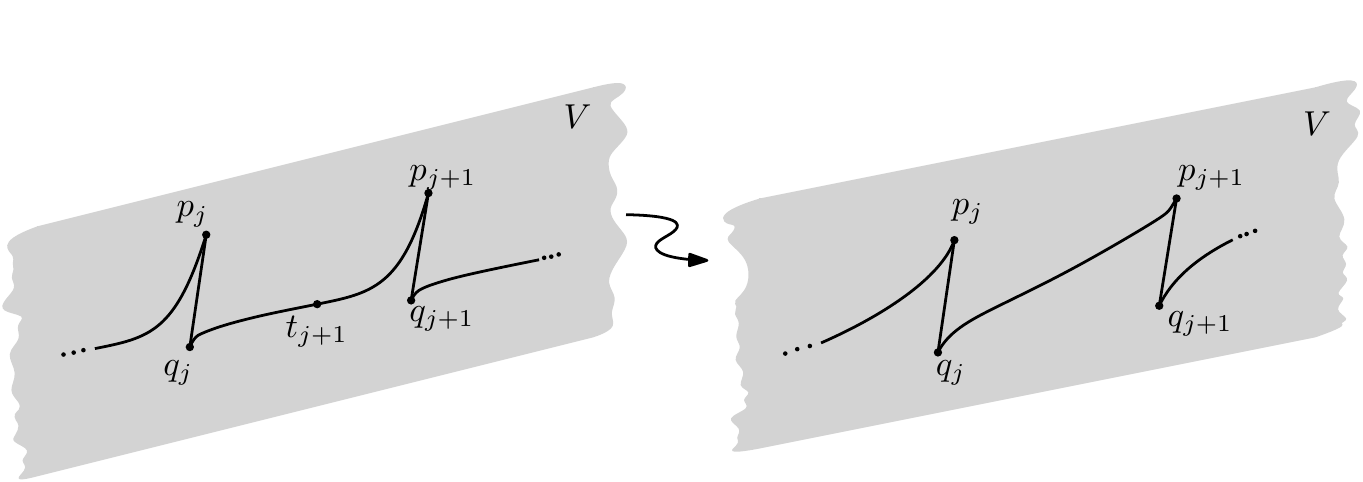}
    \caption{We increase the slope on interval $[q_j,p_{j+1}]$ so it becomes close to $\frac{z_1(p_{j+1})-z_1(q_j)}{\varphi_1(p_{j+1})-\varphi_1(q_j)}<\delta_{m+1}$}
    \label{StepIII}
\end{figure}

\textbf{Step III (increasing the slope):} We increase the slope in the front projection on each of the intervals $[q_i,p_{i+1}]$, so that we end up with the isotropic curve $\gamma_3(t)=(z_1(t),\varphi_2(t),\rho_2(t),G(t))$ which satisfies $\rho_2(t)=\frac{z_1'(t)}{\varphi_2'(t)}<\delta_{m+1}$ for all $t\in[0,1]$. This isotopy is supported inside $\bigsqcup_{i=0}^{N-1}\pi_z^{-1}(Op([q_i,p_{i+1}]))\cap V$ so it can be extended to a contact diffeomorphism $\Psi_3$ which satisfies $d_{C^0}(\Psi_3,\mathrm{Id})<\delta_{m+1}$ (see figure \ref{StepIII}).\\

\textbf{Step IV (dilatation of $\varphi$):} Let $\lambda:[0,1]\rightarrow[0,1]$ be a smooth function which satisfies
\begin{itemize}
    \item $\lambda(t)=1$ for $t\in Op(\{q_i\})$, and $\lambda(t)=\frac{\varepsilon_m}{\varepsilon_{m+1}}$ for $t\in Op(\{p_i\})$,
    \item $\lambda'(t)\leq 0$ for $t\in[p_i,q_i]$, and $\lambda'(t)\geq 0$ for $t\in[q_i,p_{i+1}]$.
\end{itemize}
Note that $\lambda'(t)\cdot\varphi'_2(t)\geq 0$, therefore we have
\[\widetilde{\rho_s}(t):=\frac{z_1'(t)}{((1-s+s\lambda(t))\cdot\varphi_2(t))'}\leq\frac{z'_1(t)'}{\rho'_2(t)(1-s+s\lambda(t))}\leq\rho_2(t)<\delta_{m+1}.\]
This means that the isotopy $(z_1(t),\rho_2(t)(1-s+s\lambda(t)),\widetilde{\rho_s},G(t))$ is supported inside $\bigsqcup_{i=0}^{N-1}(q_i,q_{i+1})\times B^{2n}(\delta_{m+1})$, thus it gives rise to a contact diffeomorphism $\Psi'_4$ which satisfies $d_{C^0}(\Psi'_4,\mathrm{Id})<\delta_m$. Similarly, we define $\hat{\rho}_s(t):=\frac{z_1'(t)}{(((1-s)\lambda(t)+s\varepsilon_m/\varepsilon_{m+1})\cdot\varphi_2(t))'}<\delta_{m+1}$ and consider isotopy of isotropic embeddings $(z_1(t),((1-s)\lambda(t)+s\varepsilon_m/\varepsilon_{m+1})\varphi_2(t),\hat{\rho}_s(t),G(t))$ which gives us contact diffeomorphism $\Psi''_4$ such that $d_{C^0}(\Psi''_4,\mathrm{Id})<\delta_{m}$. We end up with the isotropic curve (see figure \ref{StepsIVandV})
\[\gamma_{4}(t)=\Psi''_4\circ\Psi'_4\circ\gamma_3(t)=\left(z_1(t),\frac{\varepsilon_m}{\varepsilon_{m+1}}\varphi_2(t),\frac{\varepsilon_{m+1}}{\varepsilon_m}\rho_2(t),G(t)\right).\]

\begin{figure}[h]
    \centering
    \includegraphics[scale=0.5]{ 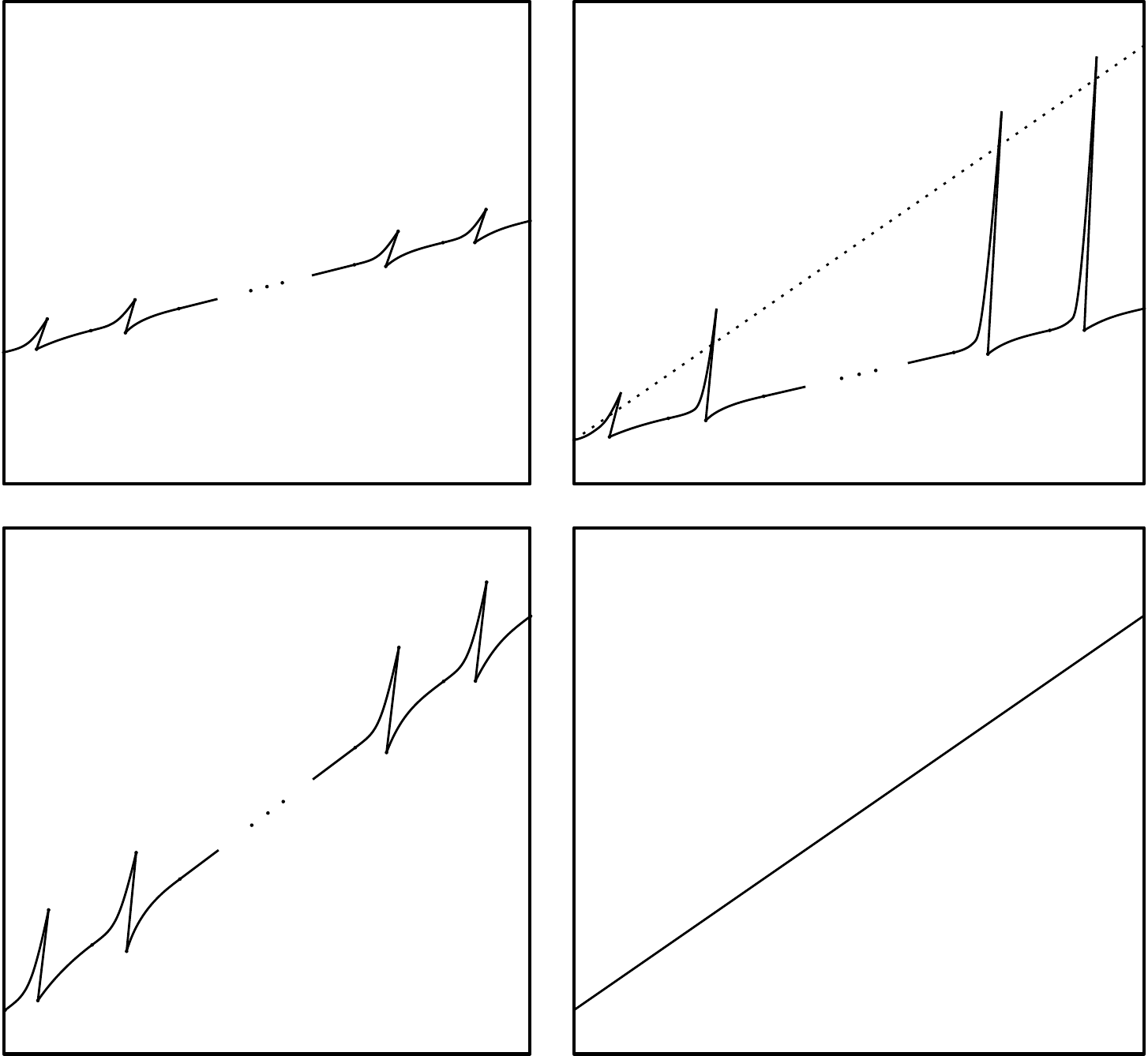}
    \caption{Up left is $\gamma_3$, down left is $\gamma_4=T\circ\gamma_3$ and down right is $\gamma_{k_{m+1}}$}
    \label{StepsIVandV}
\end{figure}

\textbf{Step V (removing cusps):} Define the transformation \[T(z,\varphi,\rho,x_2,y_2,\ldots,x_n,y_n)=\left(z,\frac{\varepsilon_m}{\varepsilon_{m+1}}\varphi,\frac{\varepsilon_{m+1}}{\varepsilon_m}\rho,x_2,y_2,\ldots,x_n,y_n\right).\]

Note that $T\circ\gamma_3=\gamma_4$ and $T\circ\gamma_{m}=\gamma_{m+1}$. We now repeat all 4 steps, but at every step we apply the transformation $T$ to each isotopy. Composition of these transformed isotopies maps $\gamma_{m+1}=T\circ\gamma_{m}$ to $\gamma_4=T\circ\gamma_3$. Let $\Phi$ be the composition of all contact diffeomorphisms obtained by extending transformed isotopies. Then $d_{C^0}(\Phi,\mathrm{Id})<6\delta_{m}$ and $\Phi\circ\gamma_{m+1}=\gamma_4$. Finally, we finish the proof by defining $\psi''_m=\Phi^{-1}\circ\Psi''_4\circ\Psi'_4\circ\Psi_3\circ\Psi''_2\circ\Psi'_2\circ\Psi_1$.\Qed

\subsection{Closed case}

We prove Theorem \ref{Theorem2} (b). More precisely, we pick any transverse knot in $(V,\xi)$ and find a contact homeomorphism, supported in tubular neighbourhood of a given transverse knot, which maps some isotropic knot to the transverse one.\\

By the contact neighborhood theorem, a sufficiently small neighbourhood of a transverse knot can be contactly embedded into an open neighbourhood of $S^1\times\{0\}^{2n}\subset S^1\times\mathbb{R}^{2n}$ with the contact form $\alpha=d\theta+x_1dy_1-y_1dx_1+\sum_{j=2}^nx_jdy_j$. Denote the transverse knot \[\eta:S^1=\mathbb{R}/(2\pi\mathbb{Z})\rightarrow S^1\times\mathbb{R}^{2n},\quad t\mapsto(\theta=t,0,0,\ldots,0).\]

Let $\gamma_m(t)=(t,\frac{1}{m}\sin(m^2t),\frac{1}{m}\cos(m^2t),0,\ldots,0)$ be a sequence of isotropic knots which $C^0$-converges to $\eta$.\\

Define surface $S_m:S^1\times[0,1]\rightarrow S^1\times\mathbb{R}^{2n}$ by
\[S_m(t,s)=(1-s)\gamma_{m}(t)+s\eta(t).\]
Note that $S_m$ is an embedding and moreover $S_m(t,0)=\gamma_m(t)$ and $S_m(t,1)=\eta(t)$. 

\begin{lemma}\label{StrechingClosedCase}
    Let $U\supset\gamma_m(S^1)$ be open neighbourhood. Then there exists $s_0\in(0,1)$ and a contact diffeomorphism $\Phi_m$ supported in $Op(S_m(S^1\times[0,1]))$ such that
    \begin{itemize}
        \item $\Phi_m\circ S_m(t,s_0)=\eta(t)$ for each $t\in S^1$,
        \item $\mathrm{diam}\,\Phi_m\circ S_m(\{t\}\times[0,s_0])<\frac{2}{m}$ for each $t\in S^1$,
        \item $d_{C^0}(\Phi_m,\mathrm{Id})<\frac{2}{m}$.
    \end{itemize}
\end{lemma}

\begin{proof}
    Pick $s_0\in(0,1)$ so that $S_m(t,s_0)\in U$ for all $t$. Let $\lambda:S^1\rightarrow[0,1]$ be a smooth function such that the diameter of each connected component of supports of $\lambda$ and $1-\lambda$ is less than $\varepsilon>0$. Consider transverse isotopy 
    \[\Gamma'_s(t):=\left(t,\frac{s_0(1-s\lambda(t))}{m}\sin(m^2t),\frac{s_0(1-s\lambda(t))}{m}\cos(m^2t),0,\ldots,0\right),\quad s\in[0,1].\]
    Note that $\alpha(\frac{d}{dt}\Gamma'_s(t))=1-s_0^2(1-s\lambda(t))^2>0$. The isotopy $\Gamma'_s$ can be realised by a contact isotopy $\Phi'_s$ supported in arbitrarily small neighbourhood of $\mathrm{supp}(\lambda)\times B^{2n}(1/m)$. Note that $\Gamma'_0(t)=S_m(t,s_0)$. Consider another transverse isotopy
    \[\Gamma''_s(t):=\left(t,\frac{(1-s)s_0(1-\lambda(t))}{m}\sin(m^2t),\frac{(1-s)s_0(1-\lambda(t))}{m}\cos(m^2t),0,\ldots,0\right),\,s\in[0,1].\]
    Note that $\alpha(\frac{d}{dt}\Gamma''_s(t))=1-(1-s)^2s_0^2(1-\lambda(t))^2>0$. The isotopy $\Gamma''_s$ be realised by a contact isotopy $\Phi''_s$ supported in arbitrarily small neighbourhood of $\mathrm{supp}(1-\lambda)\times B^{2n}(1/m)$. Note that $
    \Gamma''_0(t)=\Gamma'_1(t)$ and $\Gamma''_1(t)=\eta(t)$. Finally, define $\Phi_m=\Phi''_1\circ\Phi'_1$, and if $\varepsilon$ is small enough we will have $d_{C^0}(\Phi_m,\mathrm{Id})<\frac{2}{m}$.
\end{proof}

Let $k_1\in\mathbb{N}$ and let $\{U_i\}_{i=1}^{\infty}$ and $\{W_i\}_{i=1}^{\infty}$ be decreasing sequences of open sets, such that $\bigcap_{i=1}^{\infty} U_i=\gamma_{k_1}(S^1)$ and $\bigcap_{i=1}^{\infty} W_i=\eta(S^1)$. Using induction, we will construct a sequence of contact diffeomorphisms $\{\psi_i\}_{i=1}^{\infty}$ and an increasing sequence of integers $\{k_i\}_{i=1}^{\infty}$ such that $k_i>2^i$ and
\begin{enumerate}[label=($\mathcal{I}$\arabic*)]
    \item $\psi_i$ is supported in $S^1\times B^{2n}(2^{-i})\subset\varphi_{i-1}(U_i)\cap W_i$, where $\varphi_i:=\psi_i\circ\psi_{i-1}\circ\cdots\circ\psi_1$,
    \item $\psi_i\circ\gamma_{k_i}=\gamma_{k_{i+1}}$,
    \item $d_{C^0}(\psi_i,\mathrm{Id})<\frac{4}{2^i}$.
\end{enumerate}

Suppose we have constructed $\psi_1,\ldots,\psi_{m-1}$, so we now proceed with the induction step. Let $U:=\varphi_{m-1}(U_{m+1})$. We apply Lemma \ref{StrechingClosedCase} for $\gamma_{k_m}=\varphi_{m-1}\circ\gamma_{k_1}$ and its open neighbourhood $U$, and as a result we get a contact diffeomorphism $\psi'_m$ which satisfies $\eta(S^1)\subset\psi'_m(U)$. Let $r>0$ be small enough so that $S^1\times B^{2n}(r)\subset\psi'_m(U)\cap W_m$. Let $k_{m+1}>\max\{1/r,2^{m+1}\}$, and $s_1\in(0,1)$ such that $\psi'_m\circ S_m(S^1\times\{s_1\})\subset S^1\times B^{2n}(r)$. Consider homotopy
\[\Gamma(t,s)=
    \begin{cases}
      \psi'_m\circ S_m(t,s),& \text{for }(t,s)\in S^1\times[0,s_0],\\
      \frac{1-s}{1-s_0}\cdot \psi'_m\circ S_m(t,s_0)+\frac{s-s_0}{1-s_0}\cdot\gamma_{k_{m+1}},& \text{otherwise.}
    \end{cases}
  \]
  Note that $\Gamma(t,0)=\psi'_m\circ\gamma_{k_m}, \Gamma(t,1)=\gamma_{k_{m+1}}$, and moreover size of homotopy $\Gamma$ is less than $\frac{2}{2^m}$ if $r$ is small enough. Therefore quantitative $h$-principle (Theorem \ref{Theorem1}) gives us contact diffeomorphism  $\psi''_m$ such that $d_{C^0}(\psi''_m,\mathrm{Id})<\frac{2}{2^m}$, and $\psi''_m\circ\psi'_m\circ\gamma_{k_m}=\gamma_{k_{m+1}}$. Finally, we can define $\psi_m:=\psi''_m\circ\psi'_m$, and we are done with the induction.\\
  
  As a result of induction, we get a sequence of contact diffeomorphisms $\{\varphi_i\}_{i\geq 1}$ which satisfy properties that we need in order to apply the Lemma \ref{homeomorphismLemma} which tells us that the sequence $C^0$-converges to a homeomorphism $h$ which satisfies $h\circ\gamma_{k_1}=\eta$.\Qed
  
\section{A $C^0$-counterexample to the Lagrangian Arnold conjecture}

The construction consists of two main parts, and we now show the sketch of the construction:

\textbf{Step I)} We start with the Hamiltonian function $H(q,p)=H(q)$ and consider the image $\phi^1_H(L_0)$. Note that $\phi^1_H(L_0)\cap L_0=\mathrm{Crit}(H)$, and we can ensure that this set is finite (we identify $L$ with the zero section $L_0$). Next, we find a connected tree $T\subset L$ whose vertices correspond to $\mathrm{Crit}(H)\subset L$. Finally, we construct (with the help of quantitative $h$-principle for curves) a sequence of Hamiltonian diffeomorphisms $\{\psi_i\}$ so that $\varphi_i:=\psi_i\circ\psi_{i-1}\circ\cdots\circ\psi_1\circ\phi^1_H$ $C^0$-converges to a Hamiltonian homeomorphism $\varphi$ which satisfies $\varphi(L_0)\cap L_0=T$. We should keep in mind that the purpose of $\psi_i$ is to make $\varphi_{i-1}(T)$ closer to $T$, while making sure that we do not create additional intersections with the zero section $L_0$.

\textbf{Step II)} In the second step we construct a sequence of Hamiltonian diffeomorphisms that preserve the zero section, and contract the tree $T$ to a single point.

\subsection{Building an invariant tree}

Let $H:L\rightarrow\mathbb{R}$ be a Morse function. Let $T_0$ be any finite (connected) tree whose vertices correspond to the critical points of $H$. To every edge $e_{p,q}$ of the tree $T_0$ ($p$ and $q$ are critical points of $H$ corresponding to the vertices of the edge $e_{p,q}$) we associate a continuous function $\alpha_{p,q}:[0,1]\rightarrow L$, such that $\alpha_{p,q}(0)=p,\,\alpha_{p,q}(1)=q$, $dH|_{\alpha_{p,q}((0,1))}\neq 0$ and $\alpha_{p,q}|_{(0,1)}$ is a smooth embedding. Using a general position argument, we can additionally assume that we have $\alpha_{p,q}((0,1))\cap\alpha_{p',q'}((0,1))=\emptyset$ whenever $e_{p,q}\neq e_{p',q'}$. Denote by $T\subset L$ the image of the tree $T_0$.

\begin{prop}\label{prop1}
    There exists a Hamiltonian homeomorphism $\psi\in\mathrm{Hameo}(T^*L)$ such that $\psi|_{T}=\mathrm{Id}$ and $\psi(L_0)\cap L_0=T$ (we identify $L$ with the zero section $L_0$).
\end{prop}

\begin{proof}
Let $\alpha=\alpha_{p,q}:[0,1]\rightarrow L_0$ be one of the edges. Let $\{U_i\}_{i\geq 1}$ be a decreasing sequence of open sets with $\bigcap_{i\geq 1}U_i=\alpha([0,1])$. We define parametric surface $\Sigma_1:[0,1]^2\rightarrow T^*L$ by $\Sigma_1(t,s)=\phi^s_H(\alpha(t))$ and the curve $\gamma_1(t)=\phi_H^1(\alpha(t))$. We inductively construct:
\begin{itemize}
    \item decreasing sequence $\{\varepsilon_i\}_{i\geq 1}$, with $0<\varepsilon_i<\frac{1}{3^i}$,
    \item sequence $\{\psi_i\}_{i\geq 1}$ of compactly supported Hamiltonian diffeomorphisms in $T^*L$,
    \item sequence of parametric surfaces $\Sigma_i:[0,1]^2\rightarrow T^*L$, $i\geq 1$,
    \item decreasing sequence of open sets $\{W_i\}_{i\geq 1}$ such that $\bigcap_{i\geq 1} W_i=\alpha([0,1])$,
\end{itemize}
such that the following conditions hold:
\begin{enumerate}[label={($\mathcal{I}$\arabic*)}]
    \item $\psi_i$ is supported inside $\varphi_i(U_i)\cap W_i$, where $\varphi_i=\psi_i\circ\cdots\circ\psi_1$,
    \item $d_{C^0}(\psi_i,\mathrm{Id})<4\varepsilon_i$ and $\psi_i$ is generated by a Hamiltonian $H_i$ with $||H_i||_{\infty}<\varepsilon_i$,
    \item $\Sigma_i([0,1]^2)\subset\varphi_{i-1}(U_i)\cap W_i$,
    \item $\mathrm{diam}\,\Sigma_i(\{t\}\times[0,1])<\varepsilon_i$ for all $t\in[0,1]$,
    \item $\Sigma_i(t,0)=\alpha(t)$, and $\gamma_i(t):=\Sigma_i(t,1)=\varphi_{i-1}\circ\phi_H^1(\alpha(t))$,
    \item $\Sigma_i((0,1)^2)\cap L_0=\emptyset$, and $\varphi_i(\phi_H^1(L_0))\cap L_0=\phi_H^1(L_0)\cap L_0$.
\end{enumerate}

We need following lemma for the construction.

\begin{lemma}[streching the neighbourhood]
    There exists a Hamiltonian diffeomorphism $\psi'_m\in\mathrm{Ham}_c(T^*L)$ and a surface $\Sigma'_m:[0,1]^2\rightarrow T^*L$ such that
    \begin{enumerate}
        \item $\Sigma'_m([0,1]^2)\subset\psi'_m(\varphi_{m-1}(U_m)\cap W_m)$,
        \item $\Sigma'_m(t,i)=\Sigma_m(t,i)$ for all $(t,i)\in[0,1]\times\{0,1\}$, and $\Sigma'((0,1)^2)\cap L_0=\emptyset$,
        \item $\mathrm{diam}\,\Sigma'_m(\{t\}\times[0,1])<\varepsilon_m$ for all $t\in[0,1]$,
        \item $d_{C^0}(\psi'_m,\mathrm{Id})<\varepsilon_m$, and $\psi'_m$ is generated by arbitrary $C^0$-small Hamiltonian.
    \end{enumerate}
\end{lemma}

\begin{figure}[h]
    \centering
    \includegraphics[scale=0.6]{ 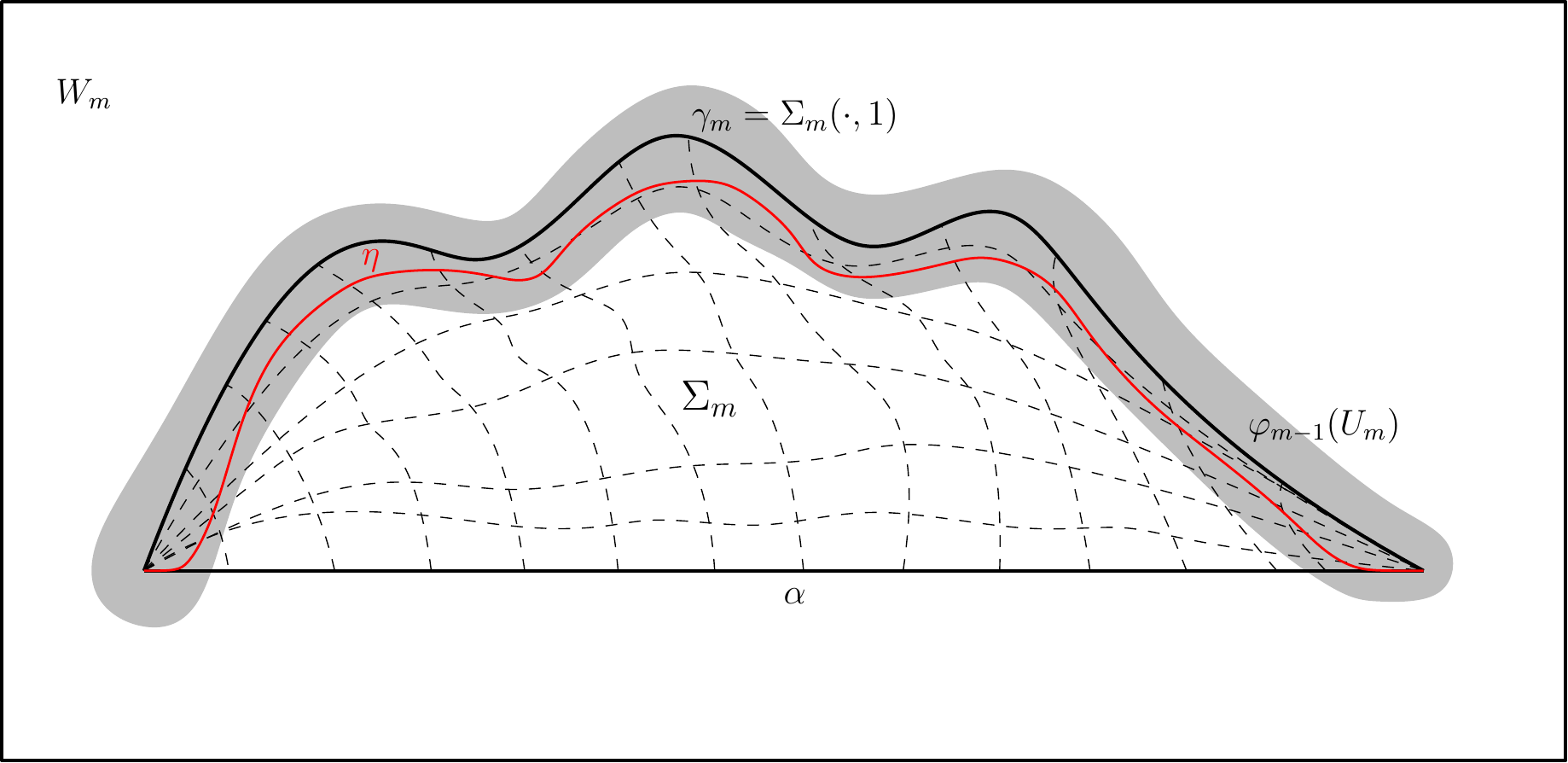}
    \caption{$\eta$ coincides with $\alpha$ near endpoints.}
    \label{CreatingCusp}
\end{figure}

\begin{proof}
    Define curve $\eta(t)=\Sigma_m(t,f(t))$, where $f:[0,1]\rightarrow[0,1)$ is a function such that $\eta([0,1])\subset\varphi_{m-1}(U_m)\cap W_m$ and $f(t)=0$ for $t\in[0,\delta]\cup[1-\delta,0]$, where $\delta>0$ is small enough. Perturb the surface $\Sigma_m$ in the neigbourhood of $\eta(1/2)$ to get the surface $\widetilde{\Sigma}_m$ such that $\int_0^1\widetilde{\eta}^*\lambda=0$, where $\widetilde{\eta}(t)=\widetilde{\Sigma}_m(t,f(t))$. We ensure that the perturbation is small enough so that $\mathrm{diam}(\widetilde{\Sigma}_m(\{t\}\times[0,1]))<\varepsilon_m$ for each $t\in[0,1]$. Note that
    \[(t,s)\mapsto\widetilde{\Sigma}_m(t,(1-s)f(t))\]
    is the isotopy between the curves $\widetilde{\eta}$ and $\alpha$ which have the same action and coincide near endpoints. Moreover, the diameter of the trajectory of any point is less than $\varepsilon_m$. Let $W=Op(\widetilde{\Sigma}_m([\delta/2,1-\delta/2]\times [0,\max f]))\subset T^*L$, such that $\gamma_m([0,1])\cap W=\emptyset$. We now apply quantitative $h$-principle for curves $\widetilde{\eta}|_{[\delta/2,1-\delta/2]}$ and $\alpha|_{[\delta/2,1-\delta/2]}$ and symplectic manifold $(W,-d\lambda)$. As a result we get the Hamiltonian diffeomorphism $\psi'_m$ supported in $W$, such that $d_{C^0}(\psi'_m,\mathrm{Id})<2\varepsilon_m$, and $\psi'_m\circ\widetilde{\eta}=\alpha$. Finally, define the surface
    \[\Sigma'_m(t,s)=\psi'_m\circ\widetilde{\Sigma}_m(t,1-s+sf(t)).\]
    
    Since $W\cap\gamma_m([0,1])=\emptyset$, we have $\Sigma'_m(t,1)=\gamma_m(t)=\Sigma_m(t,1)$. Moreover, $\psi'_m\circ\widetilde{\eta}=\alpha$, hence $\Sigma'_m(t,0)=\alpha(t)=\Sigma_m(t,0)$. We picked $\eta$ so that $\widetilde{\Sigma}_m([0,1]\times[f(t),1])\subset\varphi_{m-1}(U_m)\cap W_m$, which means that $\Sigma'_m([0,1]^2)=\psi'_m\circ\widetilde{\Sigma}_m([0,1]\times[f(t),1])\subset\psi'_m(\varphi_{m-1}(U_m)\cap W_m)$. On the other hand it can happen that $\Sigma'_m((0,1)^2)\cap L_0\neq\emptyset$. However $\mathrm{dim}(\Sigma'_m)+\mathrm{dim}(L_0)<\mathrm{dim}(T^*L)$, therefore we can make arbitrary small perturbation in the interior of the surface $\Sigma'_m$ to ensure transversality between $\Sigma'_m$ and $L_0$, which in this case means no intersections.
\end{proof}

We now describe the induction step. Let $\varepsilon_{m+1}<\min\{1/3^{m+1},\varepsilon_m\}$ and let $W_{m+1}=Op(\Sigma'_m((0,1)\times[0,\varepsilon_{m+1}]))$ such that $W_{m+1}\subset\psi'_m(\varphi_{m-1}(U_m)\cap W_m)$. Define  $\gamma'_{m+1}=\Sigma'_m(t,g(t))$, where $g:[0,1]\rightarrow(0,1]$ is a function which satisfies:
\begin{itemize}
    \item $\gamma'_{m+1}([0,1])\subset W_{m+1}$,
    \item $g(t)=1$ for $t\in[0,\delta]\cup[1-\delta,1]$,
    \item $\forall t\in[0,1]\,\mathrm{diam}\Sigma'_m(\{t\}\times[0,g(t)])<\varepsilon_{m+1}$.
\end{itemize}

First, perturb the surface $\Sigma'_m$ in the neighbourhood of point $\gamma'_{m+1}(1/2)$, to get new surface $\Sigma''_m$ which satisfies $\int_0^1\gamma_{m+1}^*\lambda=\int_0^1\gamma_m^*\lambda$, where $\gamma_{m+1}(t)=\Sigma''_m(t,g(t))$. Moreover, we require that the perturbation is small enough so that $\mathrm{diam}\Sigma''_m(\{t\}\times[0,1])<\varepsilon_m$ for each $t\in[0,1]$. Note that 
\[(t,s)\mapsto\Sigma''_m(t,1-s+sg(t))\]
is the isotopy between curves $\gamma_m$ and $\gamma_{m+1}$ which coincide near endpoints and have the same action. Moreover, the diameter of any point under this isotopy is less than $\varepsilon_m$. Let $W=Op(\Sigma''_m([\delta/2,1-\delta/2]\times[\min g,1]))\subset\psi'_m(\varphi_{m-1}(U_m)\cap W_m)$ be an open subset such that $W\cap L_0=\emptyset$. We now apply quantitative $h$-principle for the curves $\gamma_m|_{[\delta/2,1-\delta/2]}$ and $\gamma_{m+1}|_{[\delta/2,1-\delta/2]}$ and the symplectic manifold $(W,-d\lambda)$ to get a Hamiltonian diffeomorphism $\psi''_m$ supported in $W$, such that $d_{C^0}(\psi''_m,\mathrm{Id})<2\varepsilon_m$ and $\psi''_m\circ\gamma_{m}=\gamma_{m+1}$.\\

Finally, define $\psi_m=\psi''_{m+1}\circ\psi'_m$ and $\Sigma_{m+1}(t,s)=\Sigma''_m(t,sg(t))$. It is straightforward to check that $\psi_m$ and $\Sigma_{m+1}$ satisfy properties $(\mathcal{I}1)$-$(\mathcal{I}8)$, so we are done with the induction.\\

As a result of induction, sequence $\{\psi_i\}$ satisfies all required properties for the Lemma \ref{homeomorphismLemma}, which gives us the sequence $\{\varphi_i=\psi_i\circ\cdots\circ\psi_1\}_{i=1}^{\infty}$ which $C^0$-converges to a homeomorphism $\varphi_{\alpha}$ such that $\varphi_{\alpha}\circ\phi^1_H(\alpha(t))=\alpha(t)$.

Now let us prove that $\varphi_{\alpha}$ is the Hamiltonian homeomorphism. Note that $\psi_i=\phi^1_{H_i}$ for a Hamiltonian $H_i$ satisfying $||H_i||_{\infty}<\frac{\varepsilon_1}{3^{i-1}}$. Without loss of generality, we may assume that $H_i(t,x)$ vanishes for $t$ in the complement of an open subinterval of $[0,1]$. Let $\tau_0=0$ and $\tau_i=\sum_{k=1}^i\frac{1}{2^k}$. We define sequence of Hamiltonians $K_i$ defined as concatenations of time-reparametrizations of the $H_i$'s as follows:

\[K_i(t,x)=\left\{\begin{array}{lr}
    2^{k+1}H_k(2^{k+1}(t-\tau_k),x), & \text{for } k=0,1,\ldots,i,\,t\in[\tau_k,\tau_{k+1}]\\
    0, & \text{for } t\in[\tau_{i+1},1].
    \end{array}\right.\]
        
Each Hamiltonian $K_i$ generates the smooth isotopy $\varphi^t_i$ given by:
\[\varphi^t_i=\left\{\begin{array}{lr}
    \phi^{2^{k+1}(t-\tau_k)}_{H_k}\circ\psi_{k-1}\circ\cdots\circ\psi_1, & \forall k=0,1,\ldots,i,\,t\in[\tau_k,\tau_{k+1}]\\
    \psi_i\circ\cdots\circ\psi_1, & \forall t\in[\tau_{j+1},1].
    \end{array}\right.\]

Note that $\varphi^1_i=\varphi_i$, and define for $t\in[0,1)$ $\varphi^t_{\alpha}=\varphi^t_i$ for any $i\in\mathbb{N}$ such that $\tau_{i+1}\geq t$. We set $\varphi^1_{\alpha}=\varphi_{\alpha}$. It follows that $\varphi^t_{\alpha}\circ(\varphi^t_i)^{-1}=\mathrm{Id}$ for $i$ large enough and $t\in[0,1)$, which shows that $\varphi^t_i$ $C^0$-converges to $\varphi^t_{\alpha}$. The same argument shows that $(\varphi^t_{i})^{-1}$ converges to $(\varphi^t_{\alpha})^{-1}$, which implies that $\varphi_{\alpha}^t$ is a homeomorphism for all $t\in[0,1)$, but we already showed it is homeomorphism for $t=1$.\\

The condition $||H_i||_{\infty}\leq\frac{1}{3^i}$ implies that the sequence of Hamiltonians $K_i$ converges uniformly to the continuous function $K_{\alpha}(t,x)=2^{k+1}H_{k}(2^{k+1}(t-\tau_k),x)$ for $t\in[\tau_k,\tau_{k+1}]$, and $K_{\alpha}(1,x)=0$. This implies that $\varphi^t_{\alpha}$ is a hameotopy and therefore $\varphi_{\alpha}$ is a Hamiltonian homeomorphism.\\

To finish the proof we define $\psi=\phi^1_{\chi\cdot H}\circ\varphi_{\alpha_1}\circ\varphi_{\alpha_2}\circ\cdots\circ\varphi_{\alpha_m}$, where $\alpha_1,\ldots,\alpha_m$ are edges of a tree $\mathcal{T}$ and $\chi$ is a cut-off function which equals $1$ on $\bigcup_{t\in[0,1]}\phi^t_H(L_0)$.
\end{proof}

\subsection{Contracting tree to a point}

We now proceed to the second step, where we map an embedded tree to a single point, while keeping the zero section invariant.

\begin{lemma}\label{contractionLemma}
    For every $\varepsilon,\delta>0$ and every open neighbourhood $U\subset L$ of $T$, there exists a Hamiltonian function $H:[0,1]\times T^*L\rightarrow\mathbb{R}$ supported in $T^*_{\delta}U$ such that $\mathrm{diam}(\phi^1_H(T))\leq\varepsilon$, $||H||_{\infty}\leq\varepsilon$ and $\phi^1_H$ preserves the zero-section $L_0$.
\end{lemma}

\begin{proof}
    Because $T$ is a contractible set and all edges are smoothly embedded, we can find a contractible open set $V$ such that $T\subset V\subset\overline{V}\subset U$. Let $B\subset V$ be a ball of radius $\leq\varepsilon$. Since $V$ is a disc-like neighbourhood of $T$, there exists a smooth (time-dependent) vector field $X$ supported in $U$ whose time-one map sends $V$ into $B$.\\

    The Hamiltonian function $(q,p)\mapsto\langle p,X(q)\rangle$ vanishes on $L_0$ and its flow is supported in $T^*U$. By multiplying it with an appropriate cutoff function which equals 1 on a neighborhood of the support of $X$ in $T^*L$, we obtain a Hamiltonian $H_1$ supported in $T^*_{\delta}U$. This Hamiltonian $H$ vanishes on $L_0$, thus its flow preserves it. Moreover, by construction, the restriction of its flow to the zero section coincides with the flow of $X$.
\end{proof}

Let $\{\delta_i\}_{i\geq 0}$ be a decreasing sequence of real numbers converging to $0$. Let $\{W_i\}_{i\geq 0}\subset L_0$ be a sequence of contractible open neighbourhoods of $T\subset L_0$, such that $\overline{W_{i+1}}\subset W_i$ and $\bigcap_{i\geq 0}W_i=T$. Such neighbourhoods exists since $T$ is contractible set consisting of union of smoothly embedded open intervals together with finite number of vertices. We will prove by an induction that there exists a decreasing sequence of contractible open sets $\{U_i\}_{i\geq 0}\subset L_0$, a sequence of Hamiltonian diffeomorphisms $(h_i=\phi^1_{H_i})$, a decreasing sequence $\{\varepsilon_i\}_{i\geq 0}$ of real numbers converging to $0$, and a subsequence $\{W_{k_i}\}_{i\geq 0}$ such that the following properties hold
\begin{enumerate}[label=($\mathcal{I}$\arabic*)]
    \item $\varphi_i(W_{k_i})=U_i$ and $\varphi_{i}(T^*_{\delta_{k_i}}W_{k_i})\subset T^*_{\varepsilon_i}U_i$, where $\varphi_i=h_i\circ\cdots\circ h_0$,
    \item $H_{i+1}$ is supported in $T^*_{\varepsilon_i}U_i$, and $||H_i||_{\infty}\leq 1/3^i$,
    \item $\mathrm{diam}(T^*_{\varepsilon_i}U_i)\leq 1/3^i$,
    \item $H_i|_{L_0}\equiv 0$, i.e. $h_i$ preserves the $0$-section.
\end{enumerate}

First set $U_0=W_0,\,H_0=0,\,k_0=0$, and then assume that we have constructed sequences up to the index $m$. According to the Lemma \ref{contractionLemma} applied to the tree $\varphi_m(T)$ we can find a Hamiltonian function $H_{m+1}$ supported in $T^*_{\varepsilon_m}U_m$, such that $||H_{m+1}||_{\infty}\leq 1/3^{m+1}$, and $h_{m+1}(\varphi_{m}(T))$ is included in a ball $B_m\subset T^*_{\varepsilon_m}U_m$, of diameter less than $1/3^{m+1}$. Let $k_{m+1}>k_m$ be sufficiently large so that \[h_{m+1}\circ\varphi_m(T^*_{\delta_{k_{m+1}}} W_{k_{m+1}})\subset B_m.\]

Finally, we can define $U_{m+1}=h_{m+1}\circ\varphi_m(W_{k_{m+1}})$ and $\varepsilon_{m+1}<\varepsilon_m$ so that $h_{m+1}\circ\varphi_m(T^*_{\varepsilon_{m+1}}W_{m+1})\subset B_m$. One can easily check that $\{U_i\}_{i=0}^{m+1},\{W_{k_i}\}_{i=0}^{m+1},\{H_i\}_{i=0}^{m+1}$ and $\{\varepsilon_i\}_{i=0}^{m+1}$ still have the required properties, thus by induction we obtain infinite sequences.\\

Since $\lim_{i\rightarrow\infty}\mathrm{diam}(\overline{T^*_{\varepsilon_i}U_i})=0$ and because the sequence $\{T^{*}_{\varepsilon_i}U_i\}_{i\geq 0}$ is decreasing, the intersection $\bigcap_{i\geq 0}\overline{T^*_{\varepsilon_i}U_i}$ is a single point that we will denote by $p\in L_0\subset T^*L$.\\

Consider now the sequence of maps $\varphi_i=h_i\circ\cdots h_1$. By construction, if $x\in T$ then $\varphi_i(x)\in U_i$, and therefore it converges to $p$. Moreover, for every neighbourhood $U\subset T^*L$ of $T\subset L_0\subset T^*L$, the restrictions $\varphi_i|_{M\setminus U}$ stabilize for $i$ large enough. Now we define a map $\varphi:T^*L\setminus T\rightarrow M\setminus\{p\}$ as $\varphi(x)=\varphi_{i}(x)$ for $i$ large enough. Note that $\varphi$ is a diffeomorphism which preserves the zero-section $L_0$. We now define
\[f(x)=
    \begin{cases}
      \varphi\circ\psi\circ\varphi^{-1}(x),& \text{for }x\neq p,\\
      p,& \text{for }x=p.  
    \end{cases}
  \]
Let $q\in L_0$ be a point different from $p$. Then $\varphi^{-1}(q)\in L_0\setminus T$, and so $\psi\circ\varphi^{-1}(q)\in\psi(L_0\setminus T)$. We know that $\psi(L_0)\cap L_0=T$, therefore $\psi\circ\varphi^{-1}(q)\notin L_0$, and because $\varphi$ preserves the zero-section $L_0$ we conclude that $f(q)\notin L_0$, meaning that we must have $f(L_0)\cap L_0=\{p\}$.\\

It only remains to prove that $f$ is a Hamiltonian homeomorphism. At this point our setting perfectly fits in the end of the proof of the Proposition 15. in \cite{BHS}, where we can use Claims 18. and 19. to show that $f\circ\psi^{-1}$ is a Hamiltonian homeomorphism. Now it follows that $f$ is a Hamiltonian homeomorphism, since it is a composition of Hamiltonian homeomorphisms $f\circ\psi^{-1}$ and $\psi$.\Qed

\section{Towards $C^0$ rigidity of Legendrian submanifolds}

We prove Theorem \ref{Theorem4}, based on the following result of Entov and Polterovich, which includes the concept of contact interlinking, which we now explain for the sake of completeness.\\

\textbf{Contact interlinking (\cite{EP21}).} Let $(Y,\xi=\mathrm{ker}\lambda)$ be a contact manifold. An ordered pair $(\Lambda_0,\Lambda_1)$ of disjoint Legendrian submanifolds $\Lambda_0,\Lambda_1\subset Y$ is called \textit{interlinked} if there exists a constant $\mu=\mu(\Lambda_0,\Lambda_1,\lambda)>0$ such that every bounded contact Hamiltonian $h$ on $Y$ with $h\geq c>0$ possesses an orbit of time-length $\leq\mu/c$ starting at $\Lambda_0$ and arriving at $\Lambda_1$. The pair $(\Lambda_0,\Lambda_1)$ is called \textit{robustly interlinked}, if it stays interlinked after a perturbation via $C^1$-small Legendrian isotopy.

Let $\Sigma$ be the jet space $J^1Q=T^*Q(p,q)\times\mathbb{R}(z)$ of a closed manifold $Q$ equipped with the contact form $\lambda_{\mathrm{std}}=dz-pdq$. Let $R=\partial/\partial z$ be the Reeb vector field of $\lambda$. Let $\Lambda_0=\{p=0,z=0\}$ be the zero section.

\begin{thm}[Theorem 1.5. in \cite{EP21}]\label{thmEP21}
    \begin{enumerate}[label=(\roman*)]
        \item Let $\psi$ be a positive function on $Q$, and let $\Lambda_1:=\{z=\psi(q),p=\psi'(q)\}$ be the graph of its 1-jet. Then the pair $(\Lambda_0,\Lambda_1)$ is robustly interlinked.
        \item Assume that $\Lambda_1\subset\Sigma=J^1Q$ is a Legendrian submanifold Legendrian isotopic to $\Lambda_0$, with the following property: there is a unique chord of the Reeb flow $R_t$ starting on $\Lambda_0$ and ending on $\Lambda_1$, and this chord is non-degenerate\footnote{A Reeb chord $R_tx$ with $x\in\Lambda_0$ and $y:=R_{\tau}x\in\Lambda_1$ is \textit{non-degenerate} if $D_xR_{\tau}(T_x\Lambda_0)\oplus T_y\Lambda_1=\xi_y$.}. Then the pair $(\Lambda_0,\Lambda_1)$ is interlinked.
    \end{enumerate}
\end{thm}

We now prove stronger statement which implies Theorem \ref{Theorem4}.

\begin{prop}
    Let $(Y,\xi)$ be a contact manifold and $\Lambda\subset Y$ closed Legendrian submanifold. Assume we can find a sequence of contactomorphisms $\{\varphi_i\}_{i=1}^{\infty}$ that $C^0$-converges to a continuous map $\varphi:Y\rightarrow Y$ which satisfies $\varphi(Y\setminus\Lambda)\cap\varphi(\Lambda)=\emptyset$. Then $\varphi(\Lambda)$ cannot be nearly Reeb invariant.
\end{prop}

\begin{proof}
    For the sake of contradiction, assume that $\varphi(\Lambda)$ is nearly Reeb invariant. By the contact neighbourhood theorem, there exists an open neighbourhood $\mathcal{U}\supset\Lambda$ and a contact diffeomorphism $\psi:\mathcal{U}\rightarrow\psi(\mathcal{U})\subset(J^1\Lambda,\mathrm{ker}\,\lambda_{\mathrm{std}})$, such that $\Lambda$ is mapped to the zero section in $J^1\Lambda$.
    
    \begin{claim}
        There exists $m\in\mathbb{N}$ such that $\varphi(\Lambda)\subset\varphi_i(\mathcal{U})$ for $i\geq m$.
    \end{claim}
    \begin{proof}
        Assume contrary, that we can find increasing sequence $\{k_i\}_{i=0}^{\infty}$ such that $x_{k_i}\in\Lambda$ and $\varphi(x_{k_i})\notin\varphi_{k_i}(\mathcal{U})$. Let $y_{k_i}=\varphi_{k_i}^{-1}(\varphi(x_{k_i}))\in Y\setminus\mathcal{U}$. Since $Y\setminus\mathcal{U}$ and $\Lambda$ are both compact, we can find a subsequence $\{x_{k_i}\}_{i=1}^{\infty}$ (by abuse of notation we use the same notation for subsequence) such that $\lim_{i\rightarrow\infty}x_{k_i}=x\in L$ and $\lim_{i\rightarrow\infty}y_{k_i}=y\in Y\setminus\mathcal{U}$. Finally we get $\varphi(y)=\lim_{i\rightarrow\infty}\varphi_{k_i}(y_{k_i})=\lim_{i\rightarrow\infty}\varphi(x_{k_i})=\varphi(x)$, however $\varphi(y)\in\varphi(Y\setminus\Lambda)$ and $\varphi(x)\in\Lambda$ and we get a contradiction.
    \end{proof}

    Since $\varphi_i(\mathcal{U})$ converges to $\varphi(\mathcal{U})$ there exists an open set $\widetilde{\mathcal{U}}$ such that $\varphi(\Lambda)\subset\widetilde{\mathcal{U}}\subset\bigcap_{i=m}^{\infty}\varphi_i(\mathcal{U})$. Because $\varphi(\Lambda)$ is nearly Reeb invariant, we can find an open neighbourhood $\mathcal{W}\subset\widetilde{\mathcal{U}}$ and a contact form $\alpha\in\Omega^1(Op(\mathcal{W}))$ such that for each $t$ we have $\phi^t_{\alpha}(\mathcal{W})=\mathcal{W}$, where $\phi^t_{\alpha}$ is the Reeb flow of $\alpha$. Let $i_0\geq m$ be an index for which $\varphi_{i_0}(\Lambda)\subset\mathcal{W}$.
    
    Let $\widetilde{\mathcal{W}}:=\psi\circ\varphi_{i_0}^{-1}(\mathcal{W})\subset J^1\Lambda$ be an open neighbourhood of the zero section. Then the contact Hamiltonian flow with respect to $\lambda_{\mathrm{std}}$
    
    \[\Phi^t:=\psi\circ\varphi_{i_0}^{-1}\circ\phi^t_{\alpha}\circ\phi_{i_0}\circ\psi^{-1}:\widetilde{\mathcal{W}}\rightarrow\widetilde{\mathcal{W}}\]
    is generated by a contact Hamiltonian $h_t:\widetilde{\mathcal{W}}\rightarrow\mathbb{R}$ given by
    \[h_t=(\psi\circ\varphi_{i_0}^{-1})^*\lambda_{\mathrm{std}}(R_{\alpha}(\phi^t_{\alpha}\circ\varphi_{i_0}\circ\psi^{-1}))=f\alpha(R_{\alpha}(\phi^t_{\alpha}\circ\varphi_{i_0}\circ\psi^{-1}))=f\circ\phi^t_{\alpha}\circ\varphi_{i_0}\circ\psi^{-1},\]
    where $f:Op(\mathcal{W})\rightarrow(0,+\infty)$ is the function which satisfies $(\psi\circ\varphi_{i_0}^{-1})^*\lambda_{\mathrm{std}}=f\alpha$. Let $c=\min_{z\in\mathcal{W}}f(z)>0$. Finally, we extend $h_t$ to entire $J^1\Lambda$ so that it stays bounded and so that $h_t\geq c>0$. We get the contradiction with the Theorem \ref{thmEP21} (a), because the Hamiltonian flow of $h_t$ preserves bounded set $\widetilde{W}$, and hence it does not posses any orbit starting at $\Lambda_0$ and arriving to $\Lambda_1:=\{z=C,p=0\}$ for a constant $C$ large enough.
\end{proof}

It only remains to prove Proposition \ref{prop1.3}.

\begin{proof}[Proof of Proposition \ref{prop1.3}]
    Let $U\subset Y$ be an open neighbourhood of $K$. Theorem A in \cite{RS22} implies that there exists a transverse knot $T\subset (U,\xi)$ and a contact isotopy $\varphi^t:U\rightarrow U$ that squeezes $K$ onto $T$. In particular $\varphi^t(K)$ lays inside arbitrarily small neighbourhood of $T$ when $t>0$ is large enough. Since $T$ is nearly Reeb invariant, there exists an open neighbourhood $W\subset U$ of $T$ which is invariant under some contact form $\alpha$ generating $\xi$. Let $\tau>0$ be large enough so that $\varphi^{\tau}(K)\subset W$. Finally, we get that $(\varphi^{\tau})^{-1}(W)\subset U$ is an open neighbourhood of $K$ that is invariant under the Reeb flow of $((\varphi^{\tau})^{-1})^*\alpha$.
\end{proof}

\newpage

\appendix
\section{Appendix}
\addtocontents{toc}{\protect\setcounter{tocdepth}{0}}
\setcounter{thm}{0}
\renewcommand{\thethm}{\Alph{section}.\arabic{thm}}

\subsection{Contact structures on open manifolds}

A 1-form $\alpha$ on an odd dimensional manifold $V$ is called \textbf{contact form} if $\alpha\wedge(d\alpha)^n$ nowhere vanishes, where $\mathrm{dim}(V)=2n+1$. An \textbf{almost contact form} on a manifold $V$ is a pair $(\alpha,\omega)$ consisting of a non-vanishing 1-form $\alpha$ and a 2-form $\omega$ which is non-degenerate on $\xi=\mathrm{ker}\,\alpha$ at any point. An almost contact form $(\alpha,\omega)$ is a contact form iff $\omega=d\alpha$.\\

Any hyperplane field $\xi$ on a manifold $V$ has a \textbf{curvature}, namely the bilinear pairing
\[\xi\times\xi\rightarrow TV/\xi,\quad (X_p, Y_p)\rightarrow [X,Y]_p/\xi.\]

A \textbf{contact structure} is a hyperplane field $\xi$ with  non-degenerate curvature (then bilinear pairing defined above is a symplectic form valued in the line bundle $TV/\xi$). An \textbf{almost contact structure} is a hyperplane field $\xi$ given with an auxiliary non-degenerate pairing $\xi\times\xi\rightarrow TV/\xi$. We denote by $\mathcal{S}^+_{\mathrm{cont}}(V)$ the space of all cooriented almost contact structures on $V$, and by $\mathbb{S}^+_{\mathrm{cont}}(V)$ the space of all cooriented contact structures on $V$. 

\begin{thm}[\cite{EM02} Theorem 10.3.2]\label{thm10.3.2}
    For any open manifold $V$ the embedding
    \[\mathbb{S}^+_{\mathrm{cont}}(V)\hookrightarrow\mathcal{S}^+_{\mathrm{cont}}(V)\]
    is a homotopy equivalence. In particular, if two contact structures $\xi_0,\xi_1\in\mathbb{S}^+_{\mathrm{cont}}$ are homotopic in $\mathcal{S}^+_{\mathrm{cont}}$, then $\xi_0$ and $\xi_1$ are homotopic in $\mathbb{S}^+_{\mathrm{cont}}$.
\end{thm}

\begin{rem}
    Theorem also holds in a relative version, when one wants to extend a contact structure from a neighbourhood of a subcomplex of codimension $>1$.
\end{rem}

\subsection{Contact and isocontact embeddings}

Embedding $f:V\rightarrow(W,\xi)$ is called \textbf{contact} if it induces a contact structure on $V$. Note that $df(TV)\cap\xi$ consists of symplectic subspaces with respect to conformal symplectic class $\mathrm{CS}(\xi)$. Monomorphism (fiberwise injective homomorphism) $F:TV\rightarrow TW$ is called \textbf{contact} if $F^{-1}(\xi)$ is codimension 1 distribution on $V$ and $F(TV)\cap\xi$ consists of symplectic subspaces of $\xi$ with respect to $\mathrm{CS}(\xi)$.\\

If the manifold $V$ itself has a contact structure, then we may consider \textbf{isocontact} embeddings $f:(V,\xi_V)\rightarrow(W,\xi_W)$ which induce on $V$ the given structure $\xi_V$. If $\xi_V=\mathrm{ker}\,\alpha_V$ and $\xi_W=\mathrm{ker}\,\alpha_W$, then equivalently we can say that $f$ is isocontact if $f^*\alpha_W=\varphi\alpha_V$ for a non-vanishing function $\varphi:V\rightarrow\mathbb{R}$. Monomorphism $F:TV\rightarrow TW$ is called \textbf{isocontact} if $\xi_V=F^{-1}(\xi_W)$ and $F$ induces a conformally symplectic map $\xi_{V}\rightarrow\xi_{W}$ with respect to $\mathrm{CS}(\xi_V)$ and $\mathrm{CS}(\xi_W)$. It is important to notice that the contact condition is open, while the isocontact one is not.

\subsection{Approximate integration of tangential homotopies}

Let $\pi:\mathrm{Gr}_n W\rightarrow W$ be a Grassmanian bundle of $n$-planes tangent to $W$, and let $V$ be a $n$-dimensional manifold. Given a monomorphism $F: TV\rightarrow TW$ we will denote by $GF$ the corresponding tangential (Gauss) map $GF:V\rightarrow\mathrm{Gr}_nW$. Assume that $f_t:V\rightarrow W,\,t\in[0,1]$ is an isotopy of embeddings. We also assume that the manifolds $W$ and $\mathrm{Gr}_nW$ are endowed with Riemannian metrics. A \textbf{tangential homotopy} is a pair $(f,G_s)$ where $f:V\rightarrow W$ is an embedding, and $G_s:V\rightarrow\mathrm{Gr}_n W,\,s\in[0,1]$ is a homotopy which satisfies $G_0=df$ and $\pi\circ G_s=f$.

\begin{thm}[\cite{EM02} Theorem 4.4.1.]\label{thm4.4.1}
    Let $K\subset V$ be a polyhedron of positive codimension and $(f_t,G_{t,s})$ an isotopy of tangential homotopies. Then for any $\varepsilon,\delta>0$ there exists an isotopy of $\delta$-small diffeotopies $h^{\tau}_t:V\rightarrow V,\,\tau\in[0,1]$, and an isotopy of embeddings $rel\,\{0,1\}$
    \[\widetilde{f}_{t,s}:Op(h^1_t(K))\rightarrow W,\quad\widetilde{f}_{t,0}=f_t\]
    such that the homotopy $Gd\widetilde{f}_{s,t}$ is $\varepsilon$-close to the $G_{t,s}|_{Op(h^1_s(K))}$.  
\end{thm}

\begin{rem}
    The relative version of the theorem is also true.
\end{rem}

\subsection{Directed embeddings of open manifolds}

Let $A\subset\mathrm{Gr}_nW$ be an arbitrary subset. An embedding $f:V\rightarrow\mathrm{Gr}_nW$ is called $A$\textbf{-directed} if $Gdf(V)\subset A$. A \textbf{formal $A$-directed embedding} is a pair $(f,F_s)$ where $f:V\rightarrow W$ is an embedding, and $F_s: TV\rightarrow TW$ is a homotopy of monomorphisms covering $f$ such that $F_0=df$ and $GF_1(V)\subset A$. Given an open manifold $V$ a polyhedron $K\subset V$ is called a \textit{core} of $V$ if for an arbitrarily small neighborhood $U$ of $K$ there exists a fixed on $K$ isotopy $\varphi_t : V\rightarrow V$ which brings $V$ to $U$.

\begin{thm}[\cite{EM02} Theorem 4.5.1.]\label{thm4.5.1}
    Let $A\subset\mathrm{Gr}_n W$ be an open subset. Let $(f_t,G_{t,s})$ be a formal $A$-directed isotopy of tangential homotopies, such that $f_0$ and $f_1$ are $A$-directed. Then, there exists an isotopy of embeddings $f_{t,s}:V\rightarrow W$ $rel\,\{0,1\}$, such that $\widetilde{f}_t:=f_{t,1}$ is $A$-directed . Moreover, given a core $K\subset V$, the isotopy $f_{t,s}$ can be chosen arbitrarily $C^0$-close to $f_t$ on $Op(K)$.
\end{thm}

\begin{proof}
    Using the Theorem \ref{thm4.4.1} we can approximate $G_{t,s}$ along $h^1_t(K)$ by an isotopy $\widetilde{f}_{t,s}:Op(h^1_s(K))\rightarrow W$. Since $A$ is open subset, for a sufficiently close approximation the image $Gd\widetilde{f}_{t,1}(V)$ belongs to $A$. Additionally, we can assume that $h^{\tau}_t\equiv\mathrm{Id}$ for $t$ near $0$ and $1$ (because $f_0$ and $f_1$ are $A$-directed). Using the fact that $K$ is a core, there exists a family of isotopies $\varphi^s_t:V\rightarrow V$ such that $\varphi^1_t(V)=Op(h^1_t(K))$. We finish the proof by defining $f_{t,s}=\widetilde{f}_{t,s}\circ\varphi^s_t$.
\end{proof}

\begin{rem}
    The relative version for $(V,V_0)$ is false in general, but is true if $\mathrm{Int}V\setminus Op(V_0)$ has no compact connected components. The following version of the above theorem will be useful in proving $h$-principle for isocontact embeddings.
\end{rem}

\begin{thm}[\cite{EM02} Theorem 4.5.2.]\label{thm4.5.2}
    Let $A\subset\mathrm{Gr}_n W$ be an open subset, and $(f_t,F_{t,s})$ a formal isotopy of $A$-directed embeddings such that $f_0$ and $f_1$ are $A$-directed and $F_{t,0}\equiv df_0,\,F_{t,1}\equiv df_1$. Then there exists an isotopy of embeddings $f_{t,s}$ $rel\,\{0,1\}$ such that  $\widetilde{f}_t:=f_{t,1}$ is $A$-directed, and $F_{t,1}$ is homotopic to $d\tilde{f}_t$ through isotopy of monomorphisms $\widetilde{F}_{t,s}:TV\rightarrow TW,\,\mathrm{bs}\widetilde{F}_{t,s}=f_{t,s}$ with $G\widetilde{F}_{t,s}(V)\subset A$ for all $t,s\in[0,1]$.
\end{thm}

\begin{proof}
    Let $f_{t,s}$ be an isotopy constructed in Theorem \ref{thm4.4.1}. It is sufficent to construct a homotopy $\Psi_{t,s}$ between $F_{t,1}|_{Op(h^1_t(K))}$ and $d\widetilde{f}_t$. Subset $A$ is open, and since $GF_{t,1}|_{Op(h^1_t(K))}$ is $C^0$-close to $Gd\widetilde{f}_t$, and moreover the whole isotopy $GF_{t,s}|_{Op(h^1_t(K))}$ is $C^0$-close to $Gdf_{t,s}$, we conclude that there exists a homotopy $G_{t,s}:V\rightarrow A,\,\mathrm{bs}G_{t,s}=f_{t,s}$ between $GF_{t,1}|_{Op(h^1_t(K))}$ and $Gd\widetilde{f}_t$. Now we can define $\Psi_{t,s}$ via $G\Psi_{t,s}=G_{t,s}$.
\end{proof}

\subsection{$h$-principle for isocontact embeddings}

Let $(V,\xi_V)$ and $(W,\xi_W)$ be contact manifolds such that $V$ is open and $n=\mathrm{dim}(V)\leq\mathrm{dim}(W)-2$. A \textbf{formal isocontact embedding} is a pair $(f,F_s)$, where $f:V\rightarrow W$ is an embedding, and $F_s:TV\rightarrow TW$ is a homotopy of monomorphisms covering $f$, so that $F_0=df$ and $F_1$ is an isocontact monomorphism.

\begin{thm}[\cite{EM02} Theorem 12.3.1.]\label{thm12.3.1}
    Let $(f_t,F_{t,s})$ be an isotopy of formal isocontact embeddings between isocontact embeddings $f_0$ and $f_1$. Then $f_t$ can be isotoped ($rel\,\{0,1\}$) to an isocontact isotopy $\tilde{f}_t$, such that $d\tilde{f}_t$ is homotopic to $F_{t,1}$ through isocontact monomomorphisms. Moreover, given a core $V_0\subset V$ one can choose the isotopy $\widetilde{f}_t$ to be arbitrarily $C^0$-close to $f_0$ in $Op(V_0)$.
\end{thm}

\begin{proof}
    For each $x\in W$, we denote by $A_x\subset\mathrm{Gr}_{n}(T_{x}W)$ the subset of $n$-dimensional linear subspaces $S\subset T_{x}W$, such that $\mathrm{dim}(S\cap(\xi_W)_x)=n-1$ and $S\cap(\xi_W)_x$ is a symplectic subspace of $(\xi_W)_x$ with respect to the conformal symplectic structure $\mathrm{CS}(\xi_W)$. We define $A_{\mathrm{cont}}=\bigcup_{x\in W}A_x\subset\mathrm{Gr}_nW$. Note that $A_{\mathrm{cont}}$ is an open subset, and note that monomorphism $F:TV\rightarrow TW$ is contact if and only if $GF(V)\subset A_{\mathrm{cont}}$.\\
    
    Formal isocontact isotopy $(f_t,F_{t,s})$ lits to a tangential isotopy $(f_t,GF_{t,s})$ such that $GF_{t,1}(V)\subset A_{\mathrm{cont}}$. Since $A_{\mathrm{cont}}$ is open, according to Theorem \ref{thm4.5.1} $f_t$ is isotopic ($rel\,\{0,1\}$) to an isotopy of \textit{contact embeddings} (=$A_{\mathrm{cont}}$-directed embeddings) $\widetilde{f}_t:V\rightarrow W$. Since $d\tilde{f}_t$ is homotopic to $F_{t,1}$, we can use Theorem \ref{thm4.5.2} to additionally arrange that $d\tilde{f}_t$ and $F_{t,1}$ are homotopic via a homotopy $\Psi_{t,s}:TV\rightarrow TW$ for which $G\Psi_{t,s}(V)\subset A_{\mathrm{cont}}$. This implies that the contact structures $F_t^{-1}(\xi_W)=\xi_{V}$ and $d\widetilde{f}^*_t\xi_W$ are homotopic in $\mathcal{S}^+_{\mathrm{cont}}$, therefore Theorem \ref{thm10.3.2} implies the existence of a family of contact structures $\xi_{t,s}$ which connect $\xi_V$ and $d\tilde{f}^*_t\xi_{W}$.
    
    \begin{lemma}[\cite{EM02} Lemma 12.3.2.]\label{lem12.3.2.}
        Let $V$ be a compact manifold with boundary, $\xi_{t,s}$ be a family of contact structures, and $f_{t}:(V,\xi_{0,t})\rightarrow (W,\xi_W)$ an isotopy of isocontact embeddings. Then, there exists a contact isotopy $f_{t,s}:V\rightarrow W$ such that $f_{t,0}=f_t$ and $f_{t,1}:(V,\xi_{t,1})\rightarrow (W,\xi_W)$ is an isotopy of isocontact embeddings.
    \end{lemma}
\end{proof}

\subsection{$h$-principle for subcritical isotropic embeddings}

Let $(W,\xi_{W})$ be a contact manifold, and $V$ a compact manifold with boundary of subcirical dimension, i.e. $\mathrm{dim}(V)\leq (\mathrm{dim}(W)-1)/2$. A \textbf{formal isotropic embedding} is a pair $(f,F_s)$, where $f:V\rightarrow W$ is an embedding, and $F_s:TV\rightarrow TW$ is a homotopy of monomorphisms covering $f$, so that $F_0=df$ and $F_1(TV)\subset\xi_W$. 

\begin{thm}[\cite{EM02} Theorem 12.4.1.]\label{thm12.4.1}
    Let $f_0,f_1:V\rightarrow(W,\xi_W)$ be isotropic embeddings, which are isotopic through formal isotropic embeddings $(f_t,F_{t,s})$. Then $f_0$ and $f_1$ are isotopic through the isotopy of isotropic embeddings $\widetilde{f}_t$, so that $d\widetilde{f}_t$ and $F_{t,1}$ are homotopic through isotropic monomorphisms.
\end{thm}

\begin{rem}
    The theorem also holds in the relative and $C^0$-dense forms.
\end{rem}

\begin{proof}
    Any isotropic homomorphism $F:TV\rightarrow TW$ extends in a homotopically canonical way to an isocontact homomorphism $\widetilde{F}:T(J^1(V))\rightarrow TW$, where $J^1(V)$ is endowed with the canonical contact structure. Therefore we can extend isotopy of formal isotropic embeddings to an isotopy of formal isocontact embeddings. The subcritical dimensional condition ensures that $\mathrm{dim}\,J^1(V)\leq\mathrm{dim}\,W-2$, and hence we can apply Theorem \ref{thm12.3.1}. Then we get the required isotropic embeddings by restricting to the $0$-section of the constructed isocontact embeddings.
\end{proof}

\newpage
\bibliographystyle{apacite}

\end{document}